\documentclass{amsart}
\usepackage[all]{xy}
\usepackage{verbatim}
\usepackage{color}
\usepackage{mathtools}
\usepackage{amsthm}
\usepackage{amssymb}
\usepackage[colorlinks=true]{hyperref}
\usepackage{footmisc}
\usepackage{stmaryrd}

\numberwithin{equation}{section}

\newtheorem{theorem}[equation]{Theorem}
\newtheorem*{theorem*}{Theorem} \newtheorem{lemma}[equation]{Lemma}

\newtheorem*{conjecture*}{Mamma Conjecture}
\newtheorem*{conjecture1*}{Mamma Conjecture (revisited)}
\newtheorem{proposition}[equation]{Proposition}
\newtheorem{corollary}[equation]{Corollary}
\newtheorem*{corollary*}{Corollary}

\theoremstyle{remark}
\newtheorem{definition}[equation]{Definition}

\newtheorem{example}[equation]{Example}

\newtheorem{notation}[equation]{Notation}

\theoremstyle{remark}
\newtheorem{remark}[equation]{Remark}

\newcommand{\cA}{{\mathcal A}}
\newcommand{\cB}{{\mathcal B}}
\newcommand{\cC}{{\mathcal C}}
\newcommand{\cD}{{\mathcal D}}

\newcommand{\cF}{{\mathcal F}}

\newcommand{\cI}{{\mathcal I}}

\newcommand{\cL}{{\mathcal L}}

\newcommand{\cO}{{\mathcal O}}
\newcommand{\cP}{{\mathcal P}}

\newcommand{\cT}{{\mathcal T}}

\newcommand{\cW}{{\mathcal W}}
\newcommand{\cX}{{\mathcal X}}
\newcommand{\cY}{{\mathcal Y}}
\newcommand{\cZ}{{\mathcal Z}}

\newcommand{\bbA}{\mathbb{A}}

\newcommand{\bbC}{\mathbb{C}}

\newcommand{\bbF}{\mathbb{F}}
\newcommand{\bbG}{\mathbb{G}}

\newcommand{\bbP}{\mathbb{P}}

\newcommand{\bbQ}{\mathbb{Q}}
\newcommand{\bbZ}{\mathbb{Z}}

\DeclareMathOperator{\id}{id}

\DeclareMathOperator{\NChow}{NChow} 
\DeclareMathOperator{\NNum}{NNum} 

\newcommand{\dgcat}{\mathrm{dgcat}} 

\newcommand{\perf}{\mathrm{perf}}

\newcommand{\dg}{\mathrm{dg}}

\newcommand{\Hom}{\mathrm{Hom}}

\newcommand{\dgHo}{\mathrm{H}^0}
\newcommand{\dgHi}{\mathrm{H}^i}

\newcommand{\Hmo}{\mathrm{Hmo}}
\newcommand{\op}{\mathrm{op}}

\newcommand{\too}{\longrightarrow}

\let\oldmarginpar\marginpar
\def\marginpar#1{\oldmarginpar{\tiny #1}}

\begin{document}

\title[Noncommutative Weil conjecture]{Noncommutative Weil conjecture}
\author{Gon{\c c}alo~Tabuada}
\address{Gon{\c c}alo Tabuada, Department of Mathematics, MIT, Cambridge, MA 02139, USA}
\email{tabuada@math.mit.edu}
\urladdr{http://math.mit.edu/~tabuada}
\thanks{The author was supported by a NSF CAREER Award}

%
\date{\today}
\abstract{In this article, following an insight of Kontsevich, we extend the famous Weil conjecture (as well as the strong form of the Tate conjecture) from the realm of algebraic geometry to the broad noncommutative setting of dg categories. As a first application, we prove the noncommutative Weil conjecture (and the noncommutative strong form of the Tate conjecture) in the following cases: twisted schemes, Calabi-Yau dg categories associated to hypersurfaces, noncommutative gluings of schemes, root stacks, (twisted) global orbifolds, connective dg algebras, and finite-dimensional dg algebras. As a second application, we provide an alternative noncommutative proof of Weil's original conjecture (which avoids the involved tools used by Deligne) in the cases of intersections of two quadrics and linear sections of determinantal varieties. Finally, we extend also the classical theory of $L$-functions (as well as the corresponding conjectures of Tate and Beilinson) from the realm of algebraic geometry to the  broad noncommutative setting of dg categories. Among other applications, this leads to an alternative noncommutative proof of a celebrated convergence result of Serre.}}
\maketitle


\section{Statement of results: Zeta functions}\label{sec:intro}
Let $k=\bbF_q$ be a finite field of characteristic $p$, $W(k)$ the ring of $p$-typical Witt vectors of $k$, and $K:=W(k)_{1/p}$ the fraction field of $W(k)$. Given a smooth proper $k$-scheme $X$ of dimension $d$, recall that its zeta function is defined as the formal power series $Z(X;t) := \mathrm{exp}(\sum_{n \geq 1}\# X(\bbF_{q^n})\frac{t^n}{n}) \in \bbQ[\![t]\!]$, where $\mathrm{exp}(t):=\sum_{n \geq 0} \frac{t^n}{n!}$. In the same vein, given an integer $0 \leq w \leq 2d$, consider the formal power series $Z_w(X;t):=\mathrm{det}(\id - t\mathrm{Fr}^w |H^w_{\mathrm{crys}}(X))^{-1} \in K[\![t]\!]$, where $H^\ast_{\mathrm{crys}}(X)$ stands for the crystalline cohomology $H^\ast_{\mathrm{crys}}(X/W(k))\otimes_{W(k)}K$ of $X$, $\mathrm{Fr}$ for the Frobenius endomorphism of $X$, and $\mathrm{Fr}^w$ for the induced automorphism of $H^w_{\mathrm{crys}}(X)$. Thanks to the Lefschetz trace formula established by Grothendieck and Berthelot (see \cite[Chapitre VII \S3.2]{Berthelot}), we have the following weight decomposition:
\begin{equation}\label{eq:factorization}
Z(X;t)=\frac{Z_0(X;t) Z_2(X;t)\cdots Z_{2d}(X;t)}{Z_1(X;t)Z_3(X;t) \cdots Z_{2d-1}(X;t)}\,.
\end{equation}

In the late forties, Weil \cite{Weil} conjectured the following\footnote{The conjecture $\mathrm{W}(X)$ is a modern formulation of Weil's original conjecture; in the late forties crystalline cohomology was not yet developed.}:

\vspace{0.1cm}

{\it Conjecture $\mathrm{W}(X)$: The eigenvalues of the automorphism $\mathrm{Fr}^w$, with $0 \leq w \leq 2d$, are algebraic numbers and all their complex conjugates have absolute value $q^{\frac{w}{2}}$.}

\vspace{0.1cm}

In the particular case of curves, this famous conjecture follows from Weil's pioneering work \cite{Weil1}. Later, in the seventies, it was proved in full generality by~Deligne\footnote{Deligne worked with \'etale cohomology instead. However, as explained by Katz-Messing in \cite{KM}, Deligne's results hold similarly in crystalline cohomology. More recently, Kedlaya \cite{Kedlaya} gave an alternative proof of the Weil conjecture which uses solely $p$-adic techniques.} \cite{Deligne}. In contrast with Weil's proof, which uses solely the classical intersection theory of divisors on surfaces, Deligne's proof makes use of several involved tools such as the theory of monodromy of Lefschetz pencils. The Weil conjecture has numerous applications. For example, when combined with the weight decomposition \eqref{eq:factorization}, it implies that the polynomials $\mathrm{det}(\id - t \mathrm{Fr}^w |H^w_{\mathrm{crys}}(X))$ have integer coefficients.

Recall that the Hasse-Weil zeta function of $X$ is defined as the (convergent) infinite product $\zeta(X;s):= \prod_{x \in X^{(d)}} (1- (q^{\mathrm{deg}(x)})^{-s})^{-1}$, with $\mathrm{Re}(s)>d$, where $X^{(d)}$ stands for the set of closed points of $X$ and $\mathrm{deg}(x)$ for the degree of the finite field extension $\kappa(x)/\bbF_q$. In the same vein, given an integer $0 \leq w \leq 2d$, consider the function $\zeta_w(X;s):= \mathrm{det}(\id - q^{-s}\mathrm{Fr}^w |H^w_{\mathrm{crys}}(X))^{-1}$. It follows from the Weil conjecture that $\zeta(X;s)=Z(X;q^{-s})$, with $\mathrm{Re}(s)>d$, and that $\zeta_w(X;s)=Z_w(X;q^{-s})$, with $\mathrm{Re}(s)>\frac{w}{2}$. Thanks to \eqref{eq:factorization}, we hence obtain the weight decomposition:
\begin{eqnarray}\label{eq:factorization1}
\zeta(X;s)= \frac{\zeta_0(X;s)\zeta_2(X;s) \cdots \zeta_{2d}(X;s)}{\zeta_1(X;s) \zeta_3(X;s) \cdots \zeta_{2d-1}(X;s)} && \mathrm{Re}(s)>d\,.
\end{eqnarray}
Note that \eqref{eq:factorization1} implies automatically that the Hasse-Weil zeta function of $X$ admits a (unique) meromorphic continuation to the entire complex plane. 
\begin{remark}[Periodicity]\label{rk:periodicity}
Note that the Hasse-Weil zeta function of $X$ is periodic in the sense that $\zeta(X;s)=\zeta(X;s + \frac{2\pi i}{\mathrm{log}(q)})$. Similarly, $\zeta_w(X;s)=\zeta_w(X;s + \frac{2\pi i}{\mathrm{log}(q)})$.
\end{remark}
\begin{remark}[Riemann hypothesis]\label{rk:RH}
The above conjecture $\mathrm{W}(X)$ is usually  called the ``analogue of the Riemann hypothesis'' because it implies that if $z\in \bbC$ is a pole of $\zeta_w(X;s)$, then $\mathrm{Re}(z)=\frac{w}{2}$. Consequently, if $z \in \bbC$ is a pole, resp. a zero, of $\zeta(X;s)$, then $\mathrm{Re}(z)\in \{0, 1, \ldots, d\}$, resp. $\mathrm{Re}(z)\in \{ \frac{1}{2}, \frac{2}{3}, \ldots, \frac{2d-1}{2}\}$. 
\end{remark}
Let $\cA$ be a smooth proper $k$-linear dg category in the sense of Kontsevich; see \S\ref{sub:dg}. Examples include the (unique) dg enhancements $\perf_\dg(X)$ of the categories of perfect complexes $\perf(X)$ of smooth proper $k$-schemes $X$ (or, more generally, of smooth proper algebraic $k$-stacks $\cX$); consult \cite{ICM-Keller,LO}. As explained in \S\ref{sub:cyclotomic} below, the topological periodic cyclic homology group $TP_0(\cA)_{1/p}$ of $\cA$ (this is a finite-dimensional $K$-vector space), resp. the topological periodic cyclic homology group $TP_1(\cA)_{1/p}$ of $\cA$, comes equipped with an automorphism $\mathrm{F}_0$, resp. $\mathrm{F}_1$, called the ``cyclotomic Frobenius''. Following Kontsevich \cite{email}, we hence define the {\em even/odd zeta function of $\cA$} as the following formal power series:
\begin{eqnarray*}
Z_{\mathrm{even}}(\cA;t)& := & \mathrm{det}(\id - t\mathrm{F}_0|TP_0(\cA)_{1/p})^{-1} \in K[\![t]\!] \\
Z_{\mathrm{odd}}(\cA;t)& := & \mathrm{det}(\id - t\mathrm{F}_1|TP_1(\cA)_{1/p})^{-1} \in K[\![t]\!] \,.
\end{eqnarray*}
Weil's conjecture admits the following noncommutative counterpart:

\vspace{0.1cm}

{\it Conjecture $\mathrm{W}_{\mathrm{nc}}(\cA)$: The eigenvalues of the automorphism $\mathrm{F}_0$, resp. $\mathrm{F}_1$, are algebraic numbers and all their complex conjugates have absolute~value~$1$,~resp.~$q^{\frac{1}{2}}$.}

\vspace{0.1cm}

The noncommutative Weil conjecture was originally invisioned by Kontsevich in his seminal talks \cite{Lefschetz,IAS}. The next result relates this conjecture with Weil's original conjecture:

\begin{theorem}\label{thm:main}
Given a smooth proper $k$-scheme $X$, we have the equivalence of conjectures $\mathrm{W}_{\mathrm{nc}}(\perf_\dg(X)) \Leftrightarrow \mathrm{W}(X)$.
\end{theorem}
Intuitively speaking, Theorem \ref{thm:main} shows that the Weil conjecture belongs not only to the realm of algebraic geometry but also to the broad noncommutative setting of dg categories. 

In contrast with the commutative world, the cyclotomic Frobenius is not induced from an endomorphism\footnote{Note that in the particular case where $\cA$ is a $k$-algebra $A$, the Frobenius map $a \mapsto a^q$ is a $k$-algebra endomorphism if and only if $A$ is commutative.} of $\cA$. Consequently, in contrast with the commutative world, it is not known if the polynomials $\mathrm{det}(\id - t\mathrm{F}_0|TP_0(\cA)_{1/p})$ and $\mathrm{det}(\id - t\mathrm{F}_1|TP_1(\cA)_{1/p})$ have integer coefficients (or rational coefficients). Nevertheless, after choosing an embedding $\iota\colon K\hookrightarrow \bbC$, we can define the {\em even/odd Hasse-Weil zeta function of $\cA$} as follows:
\begin{eqnarray*}
\zeta_{\mathrm{even}}(\cA;s)& := & \mathrm{det}(\id - q^{-s} (\mathrm{F}_0 \otimes_{K, \iota}\bbC)|TP_0(\cA)_{1/p}\otimes_{K, \iota} \bbC)^{-1} \\
\zeta_{\mathrm{odd}}(\cA;s)& := & \mathrm{det}(\id - q^{-s} (\mathrm{F}_1 \otimes_{K, \iota}\bbC)|TP_1(\cA)_{1/p}\otimes_{K, \iota} \bbC)^{-1}\,.
\end{eqnarray*}
\begin{remark}[Periodicity]
Similarly to Remark \ref{rk:periodicity}, note that the even/odd Hasse-Weil zeta function of $\cA$ is periodic of period $\frac{2\pi i}{\mathrm{log}(q)}$.
\end{remark}
\begin{remark}[Noncommutative Riemann hypothesis]\label{rk:key}
Similarly to Remark \ref{rk:RH}, the conjecture $\mathrm{W}_{\mathrm{nc}}(\cA)$ may be called the ``analogue of the noncommutative Riemann hypothesis'' because it implies that if $z \in \bbC$ is a pole of $\zeta_{\mathrm{even}}(\cA; s)$, resp. $\zeta_{\mathrm{odd}}(\cA; s)$, then $\mathrm{Re}(z)=0$, resp. $\mathrm{Re}(z)=\frac{1}{2}$ (independently of the embedding $\iota\colon K \hookrightarrow \bbC$).
\end{remark}
The next result follows from (the proof of) Theorem \ref{thm:main}:
\begin{corollary}\label{cor:main}
Given a smooth proper $k$-scheme $X$, we have the following factorization $\zeta_{\mathrm{even}}(\perf_\dg(X);s)  = \prod_{w \,\mathrm{even}}\zeta_w(X; s+\frac{w}{2})$ as well as the following factorization $\zeta_{\mathrm{odd}}(\perf_\dg(X); s) = \prod_{w \,\mathrm{odd}}\zeta_w(X; s+\frac{w-1}{2})$.
\end{corollary}
Roughly speaking, Corollary \ref{cor:main} shows that the even/odd Hasse-Weil zeta function of $\perf_\dg(X)$ may be understood as the ``weight normalization'' of the product of the Hasse-Weil zeta functions $\zeta_w(X;s)$. This leads~to~the~following~result:
\begin{corollary}\label{cor:order}
Given a smooth proper $k$-scheme $X$ and a complex number $z \in \bbC$, we have the equality $\mathrm{ord}_{s=z}\zeta_{\mathrm{even}}(\perf_\dg(X); s)  = \sum_{w\,\mathrm{even}} \mathrm{ord}_{s=z+\frac{w}{2}} \zeta_w(X;s)$ as well as the equality $\mathrm{ord}_{s=z}\zeta_{\mathrm{odd}}(\perf_\dg(X); s) = \sum_{w\,\mathrm{odd}} \mathrm{ord}_{s=z+\frac{w-1}{2}} \zeta_w(X;s)$, where $\mathrm{ord}_{s=z}f(s)$ stands for the order of a meromorphic function $f(s)$ at $s=z$.
\end{corollary}
\subsection*{$l$-adic absolute value}
Let $l\neq p$ be a prime number. Given a smooth proper $k$-scheme $X$ of dimension $d$, it is well-known that the eigenvalues of the automorphisms $\mathrm{Fr}^w, 0 \leq w\leq 2d$, are algebraic numbers and that all their $l$-adic conjugates have absolute value $1$. Let $\cA$ be a smooth proper $k$-linear dg category. Motivated by the aforementioned fact, Kontsevich also conjectured in \cite{Lefschetz,IAS} the following:

\vspace{0.1cm}

{\it Conjecture $\mathrm{W}^l_{\mathrm{nc}}(\cA)$: The eigenvalues of the automorphisms $\mathrm{F}_0$ and $\mathrm{F}_1$ are algebraic numbers and all their $l$-adic conjugates have absolute value $1$.}

\vspace{0.1cm}
 
The next result (partially) solves Kontsevich's conjecture:

\begin{theorem}\label{thm:new}
Assume that there exists an integer $C\gg 0$ (which depends on $\cA$) such that the eigenvalues of the automorphisms $\mathrm{F}_0$ and $\mathrm{F}_1$ become algebraic integers after multiplication by $q^C$. Under this assumption, the conjecture $\mathrm{W}^l_{\mathrm{nc}}(\cA)$ holds.
\end{theorem}
As explained in Remark \ref{rk:new} below, the assumption of Theorem \ref{thm:new} holds when $\cA=\perf_\dg(X)$ with $X$ a smooth proper $k$-scheme. Consult \S\ref{sec:applications1} for further examples.
\subsection*{Functional equation}

Thanks to the work of M.~Artin and Grothendieck (consult \cite{Grothendieck-Bourbaki} and the references therein), the Hasse-Weil zeta function $\zeta(X;s)$ of a smooth proper $k$-scheme $X$ of dimension $d$ is known to satisfy the functional equation
\begin{equation}\label{eq:functional1111}
\zeta(X;s)=\pm q^{\chi(X)s} \cdot q^{-\frac{\chi(X)}{2} d}\cdot \zeta(X;d-s)\,,
\end{equation}
where $\chi(X)$ stands for the Euler characteristic of $X$. Morally speaking, the equality \eqref{eq:functional1111} describes a ``symmetry'' of $\zeta(X;s)$ along the vertical line $\mathrm{Re}(s)=\frac{d}{2}$. This functional equation admits the following noncommutative counterpart:
\begin{theorem}\label{thm:functional}
Given a smooth proper $k$-linear dg category $\cA$, we have the following functional equations
\begin{eqnarray*}\label{eq:functional1}
\zeta_{\mathrm{even}}(\cA;s) & = &  (-1)^{\chi_0(\cA)} \cdot q^{\chi_0(\cA)s} \cdot \mathrm{det}(\mathrm{F}_0 \otimes_{K, \iota}\bbC) \cdot \zeta_{\mathrm{even}}(\cA;-s) \\
\zeta_{\mathrm{odd}}(\cA;s) & = & (-1)^{\chi_1(\cA)} q^{-\chi_1(\cA)(1-s)} \cdot \mathrm{det}(\mathrm{F}_1 \otimes_{K, \iota}\bbC) \cdot \zeta_{\mathrm{odd}}(\cA;1-s)\,,
\end{eqnarray*}
where $\chi_0(\cA):=\mathrm{dim}_KTP_0(\cA)_{1/p}$ and $\chi_1(\cA):=\mathrm{dim}_KTP_1(\cA)_{1/p}$.
\end{theorem}
Intuitively speaking, Theorem \ref{thm:functional} describes a ``symmetry'' of $\zeta_{\mathrm{even}}(\cA;s)$, resp. $\zeta_{\mathrm{odd}}(\cA;s)$, along the vertical line $\mathrm{Re}(s)=0$, resp. $\mathrm{Re}(s)=\frac{1}{2}$.
\begin{corollary}\label{cor:functional}
When $\cA=\perf_\dg(X)$, with $X$ a smooth proper $k$-scheme, the functional equations of Theorem \ref{thm:functional} reduce to the following functional equations 
\begin{eqnarray*}
\prod_{w\,\mathrm{even}} \zeta_w(X;s+\frac{w}{2}) &= & \pm q^{\chi_{\mathrm{even}}(X)s} \cdot \prod_{w\, \mathrm{even}} \zeta_w(X;-s+\frac{w}{2}) \\
\prod_{w\,\mathrm{odd}} \zeta_w(X;s+\frac{w-1}{2}) & = & \pm q^{\chi_{\mathrm{odd}}(X)s} \cdot \prod_{w\, \mathrm{odd}} \zeta_w(X;1-s+\frac{w-1}{2})\,,
\end{eqnarray*}
where $\chi_{\mathrm{even}}(X):=\sum_{w\,\mathrm{even}} \mathrm{dim}_KH^w_{\mathrm{crys}}(X)$ and $\chi_{\mathrm{odd}}(X):=\sum_{w\,\mathrm{odd}} \mathrm{dim}_KH^w_{\mathrm{crys}}(X)$.
\end{corollary}
\begin{remark}[Related work]
In \cite{finite} we developed a general theory of (Hasse-Weil) zeta functions for smooth proper dg categories equipped with an endomorphism. Among other applications, this theory led to a far-reaching noncommutative generalization of the results of Dwork \cite{Dwork} and Grothendieck \cite{Grothendieck-Bourbaki} concerning the rationality and the functional equation of the classical (Hasse-Weil) zeta function. 
\end{remark}
\subsection*{Strong form of the Tate conjecture}
Given a smooth proper $k$-scheme $X$ of dimension $d$ and an integer $0\leq i\leq d$, let us write $\cZ^i(X)_\bbQ/_{\!\sim\mathrm{num}}$ for the $\bbQ$-vector space of algebraic cycles of codimension $i$ on $X$ up to numerical equivalence.

In the mid sixties, Tate \cite{Tate} conjectured the following:

\vspace{0.1cm}

{\it Conjecture $\mathrm{ST}(X)$: The order $\mathrm{ord}_{s=j}\zeta(X;s)$ of the Hasse-Weil zeta function $\zeta(X;s)$ at the pole $s=j$, with $0\leq j \leq d$, is equal to $-\mathrm{dim}_\bbQ\cZ^j(X)_\bbQ/_{\!\sim\mathrm{num}}$.}

\vspace{0.1cm}

This conjecture is usually called the ``strong form of the Tate conjecture''. It holds for $0$-dimensional schemes, for curves, for abelian varieties of dimension $\leq 3$, and also for $K3$-surfaces. Besides these cases (and some other cases scattered in the literature), it remains wide open.

\vspace{0.1cm}

Given a smooth proper $k$-linear dg category $\cA$, recall from \S\ref{sub:numerical} below the definition of its numerical Grothendieck group $K_0(\cA)_\bbQ/_{\!\sim \mathrm{num}}$. The strong form of the Tate conjecture admits the following noncommutative counterpart:
 
\vspace{0.1cm}

{\it Conjecture $\mathrm{ST}_{\mathrm{nc}}(\cA)$: The order $\mathrm{ord}_{s=0}\zeta_{\mathrm{even}}(\cA;s)$ of the even Hasse-Weil zeta function $\zeta_{\mathrm{even}}(\cA;s)$ at the pole $s=0$ is equal to $- \mathrm{dim}_\bbQ K_0(\cA)_\bbQ/_{\!\sim \mathrm{num}}$.}

\begin{remark}[Alternative formulation]\label{rk:multiplicity}
Note that, by definition of the even Hasse-Weil zeta function of $\cA$, the order $\mathrm{ord}_{s=0}\zeta_{\mathrm{even}}(\cA;s)$ of the even Hasse-Weil zeta function $\zeta_{\mathrm{even}}(\cA;s)$ at the pole $s=0$ agrees with the algebraic multiplicity of the eigenvalue $q^0=1$ of the automorphism $\mathrm{F}_0\otimes_{K, \iota}\bbC$ (or, equivalently, of $\mathrm{F}_0$). Hence, the conjecture $\mathrm{ST}_{\mathrm{nc}}(\cA)$ may be alternatively formulated as follows: {\em the algebraic multiplicity of the eigenvalue $1$ of $\mathrm{F}_0$ agrees with $\mathrm{dim}_\bbQ K_0(\cA)_\bbQ/_{\!\sim \mathrm{num}}$}. This shows, in particular, that the integer $\mathrm{ord}_{s=0}\zeta_{\mathrm{even}}(\cA;s)$ is independent of the embedding $\iota\colon K \hookrightarrow \bbC$ used in the definition of $\zeta_{\mathrm{even}}(\cA;s)$.
\end{remark}
\begin{remark}[Equivalent conjectures]
As proved in Theorem \ref{thm:alternative} below, the noncommutative strong form of the Tate conjecture is equivalent to the noncommutative $p$-version of the Tate conjecture plus the noncommutative standard conjecture of type $D$. Moreover, when all smooth proper dg categories are considered simultaneously, the noncommutative strong form of the Tate conjecture becomes equivalent to the fully-faithfulness of the enriched topological periodic cyclic homology functor; consult \S\ref{sub:enriched} and \S\ref{sub:embedding} below for details.
\end{remark}
The next result relates the noncommutative strong form of the Tate conjecture with the strong form of the Tate conjecture:
\begin{theorem}\label{thm:strong}
Given a smooth proper $k$-scheme $X$, we have the equivalence of conjectures $\mathrm{ST}_{\mathrm{nc}}(\perf_\dg(X))\Leftrightarrow \mathrm{ST}(X)$.
\end{theorem}
Similarly to Theorem \ref{thm:main}, Theorem \ref{thm:strong} shows that the strong form of the Tate conjecture belongs not only to the realm of algebraic geometry but also to the broad noncommutative setting of smooth proper dg categories.
\section{Applications to noncommutative geometry}\label{sec:applications1}

Let $k=\bbF_q$ be a finite field of characteristic $p$ and $l\neq p$ a prime number. In this section, making use of Theorems \ref{thm:main}, \ref{thm:new}, and \ref{thm:strong}, we prove the noncommutative Weil conjecture(s) (i.e. $\mathrm{W}_{\mathrm{nc}}(-)$ and $\mathrm{W}^l_{\mathrm{nc}}(-)$), as well as the noncommutative strong form of the Tate conjecture, in several (interesting) cases.

\subsection*{Twisted schemes}
Let $X$ be a smooth proper $k$-scheme and $\cF$ a sheaf of Azumaya algebras over $X$. Similarly to $\perf_\dg(X)$, we can also consider the (smooth proper) dg category $\perf_\dg(X;\cF)$ of perfect complexes of $\cF$-modules.
\begin{theorem}\label{thm:twist-new}
We have the following equivalences of conjectures:
\begin{eqnarray*}
\mathrm{W}(X) \Leftrightarrow \mathrm{W}_{\mathrm{nc}}(\perf_\dg(X;\cF)) &&
\mathrm{ST}(X) \Leftrightarrow \mathrm{ST}_{\mathrm{nc}}(\perf_\dg(X;\cF))\,.
\end{eqnarray*}
Moreover, the conjecture $\mathrm{W}^l_{\mathrm{nc}}(\perf_\dg(X;\cF))$ holds.
\end{theorem}
Morally speaking, Theorem \ref{thm:twist-new} shows that in what concerns the (noncommutative) Weil conjecture and the (noncommutative) strong form of the Tate conjecture, there is no difference between schemes and~twisted~schemes.

\subsection*{Calabi-Yau dg categories associated to hypersurfaces}
Let $X \subset \bbP^n$ be a smooth hypersurface of degree $\mathrm{deg}(X) \leq n+1$. As proved by Kuznetsov in \cite[Cor.~4.1]{Kuznetsov-CY}, we have a semi-orthogonal decomposition:
\begin{equation}\label{eq:semi-hyper}
\perf(X) = \langle \cT(X), \cO_X, \ldots, \cO_X(n-\mathrm{deg}(X))\rangle\,.
\end{equation}
Moreover, the associated dg category $\cT_\dg(X)$, defined as the dg enhancement of $\cT(X)$ induced from $\perf_\dg(X)$, is a smooth proper Calabi-Yau dg category of fractional dimension $\frac{(n+1)(\mathrm{deg}(X)-2)}{\mathrm{deg}(X)}$. 
\begin{remark}[Noncommutative K3-surfaces]
In the particular case where $n=5$ and $\mathrm{deg}(X)=3$, the dg categories $\cT_\dg(X)$ are usually called ``noncommutative $K3$-surfaces'' because they share many of the key properties of the dg categories of perfect complexes of the classical (smooth proper) $K3$-surfaces. Moreover, Kuznetsov conjectured in \cite{Kuznetsov1} that $\cT(X)$ is (Fourier-Mukai) equivalent to the category of perfect complexes of a $K3$-surface if and only if $X$ is rational.
\end{remark}
\begin{theorem}\label{thm:hypersurface}
We have the following equivalences of conjectures:
\begin{eqnarray*}
\mathrm{W}(X) \Leftrightarrow \mathrm{W}_{\mathrm{nc}}(\cT_\dg(X)) &&
\mathrm{ST}(X) \Leftrightarrow \mathrm{ST}_{\mathrm{nc}}(\cT_\dg(X))\,.
\end{eqnarray*}
Moreover, the conjecture $\mathrm{W}^l_{\mathrm{nc}}(\cT_\dg(X))$ holds.
\end{theorem}
Similarly to Theorem \ref{thm:twist-new}, Theorem \ref{thm:hypersurface} shows that in what concerns the (noncommutative) Weil conjecture and the (noncommutative) strong form of the Tate conjecture, there is no difference between the hypersurface $X$ and the associated Calabi-Yau dg category $\cT_{\mathrm{dg}}(X)$.

\subsection*{Noncommutative gluings of schemes}
Let $X$ and $Y$ be two smooth proper $k$-schemes and $\mathrm{B}$ a perfect dg $\perf_\dg(X)\text{-}\perf_\dg(Y)$-bimodule. Following Orlov \cite[\S3.2]{Orlov}, we can consider the gluing $X \ominus_{\mathrm{B}}Y$ of $\perf_\dg(X)$ and $\perf_\dg(Y)$ via $\mathrm{B}$ (Orlov used a different notation). This new dg category is smooth and proper. 
\begin{theorem}\label{thm:glue}
We have the following equivalences of conjectures:
\begin{eqnarray*}
\mathrm{W}(X) +\mathrm{W}(Y) \Leftrightarrow \mathrm{W}_{\mathrm{nc}}(X \ominus_{\mathrm{B}}Y) &&
\mathrm{ST}(X) + \mathrm{ST}(Y) \Leftrightarrow \mathrm{ST}_{\mathrm{nc}}(X \ominus_{\mathrm{B}}Y)\,.
\end{eqnarray*}
Moreover, the conjecture $\mathrm{W}^l_{\mathrm{nc}}(X \ominus_{\mathrm{B}}Y)$ holds.
\end{theorem}
Intuitively speaking, Theorem \ref{thm:glue} shows that the noncommutative Weil conjecture and the noncommutative strong form of the Tate conjecture are ``additive'' with respect to gluings. This implies, in particular, that the noncommutative Weil conjecture(s) hold(s) for every noncommutative gluing and that the noncommutative strong form of the Tate conjecture holds for every noncommutative gluing of curves.

\subsection*{Root stacks}
Let $X$ be a smooth proper $k$-scheme, $\cL$ a line bundle on $X$, $\varsigma \in \Gamma(X,\cL)$ a global section, and $n\geq 1$ an integer. Following Cadman \cite[Def.~2.2.1]{Cadman}, the associated {\em root stack} is defined as the following fiber-product
$$
\xymatrix{
\cX:=\sqrt[n]{(\cL,\varsigma)/X} \ar[d]_-f \ar[r]& [\bbA^1/\bbG_m] \ar[d]^-{\theta_n} \\
X \ar[r]_-{(\cL,\varsigma)} & [\bbA^1/\bbG_m]\,,
}
$$
where $\theta_n$ stands for the morphism induced by the $n^{\mathrm{th}}$ power map on $\bbA^1$ and $\bbG_m$.
\begin{theorem}\label{thm:root}
Assume that the zero locus $D \hookrightarrow X$ of $\varsigma$ is smooth. Under this assumption, we have the following equivalences of conjectures:
\begin{eqnarray*}
\mathrm{W}(X)+ \mathrm{W}(D) \Leftrightarrow\mathrm{W}_{\mathrm{nc}}(\perf_\dg(\cX))   &&  \mathrm{ST}(X)+ \mathrm{ST}(D) \Leftrightarrow \mathrm{ST}_{\mathrm{nc}}(\perf_\dg(\cX))\,.
\end{eqnarray*}
Moreover, the conjecture $\mathrm{W}^l_{\mathrm{nc}}(\perf_\dg(\cX))$ holds.
\end{theorem}
Theorem \ref{thm:root} implies that the noncommutative Weil conjecture(s) hold(s) for every root stack and that the noncommutative strong form of the Tate conjecture holds, for example, for all those root stacks whose underlying scheme is a curve.

\subsection*{Global orbifolds}
Let $G$ be a finite group of order $n$, $X$ a smooth proper $k$-scheme equipped with a $G$-action, and $\cX:=[X/G]$ the associated global orbifold.
\begin{theorem}\label{thm:orbifolds}
Assume that $p \nmid n$ ($\Leftrightarrow 1/n \in k$). Under this assumption, we have the following implications of conjectures ($\sigma$ is a cyclic subgroup of $G$):
\begin{eqnarray}
\sum_{\sigma \subseteq G} \mathrm{W}(X^\sigma \times \mathrm{Spec}(k[\sigma])) & \Rightarrow & \mathrm{W}_{\mathrm{nc}}(\perf_\dg(\cX)) \label{eq:implication1} \\
 \sum_{\sigma \subseteq G} \mathrm{ST}(X^\sigma \times \mathrm{Spec}(k[\sigma])) & \Rightarrow & \mathrm{ST}_{\mathrm{nc}}(\perf_\dg(\cX)) \label{eq:implication2}\,.
\end{eqnarray} 
Under the stronger assumption $n|(q-1)$ ($\Leftrightarrow$ $k$ contains the $n^{\mathrm{th}}$ roots of unity ), the same implications hold with $X^\sigma \times \mathrm{Spec}(k[\sigma])$ replaced by $X^\sigma$. Moreover, the conjecture $\mathrm{W}^l_{\mathrm{nc}}(\perf_\dg(\cX))$ holds.
\end{theorem} 
Theorem \ref{thm:orbifolds} implies that the noncommutative Weil conjecture(s) hold(s) for every global orbifold. Since $\mathrm{Spec}(k[\sigma])$ is $0$-dimensional, it implies moreover that the noncommutative strong form of the Tate conjecture holds, for example, for all those global orbifolds whose underlying scheme is a curve.
\begin{remark}[McKay correspondence]
A famous conjecture of Reid asserts that the category $\perf(\cX)$ is (Fourier-Mukai) equivalent to the category of perfect complexes of  a(ny) crepant resolution $Y$ of the (singular) geometric quotient $X/\!\!/G$. Whenever this holds, the right-hand sides of  \eqref{eq:implication1}-\eqref{eq:implication2} may be replaced by the conjectures $\mathrm{W}(Y)$ and $\mathrm{ST}(Y)$, respectively. Reid's conjecture has been proved in several cases; consult, for example, the work of Bezrukavnikov and Kaledin \cite{BK}, Bridgeland, King and Reid \cite{BKR}, Kapranov and Vasserot \cite{KV}, and Kawamata \cite{Kawamata}.
\end{remark}
\subsection*{Twisted global orbifolds}
Let $G$ be a finite group of order $n$, $X$ a smooth proper $k$-scheme equipped with a $G$-action, $\cX:=[X/G]$ the associated global orbifold, and $\cF$ a sheaf of Azumaya algebras\footnote{Equivalently, $\cF$ is a $G$-equivariant sheaf of Azumaya algebras over $X$.} over $\cX$. Similarly to $\perf_\dg(\cX)$, we can also consider the dg category $\perf_\dg(\cX;\cF)$ of perfect complexes of $\cF$-modules.
\begin{theorem}\label{thm:twist}
Assume that $n|(q-1)$. Under this assumption, we have the following implications of conjectures
\begin{eqnarray*}
\sum_{\sigma \subseteq G} \mathrm{W}(Y_\sigma) \Rightarrow \mathrm{W}_{\mathrm{nc}}(\perf_\dg(\cX;\cF)) && \sum_{\sigma \subseteq G} \mathrm{ST}(Y_\sigma) \Rightarrow \mathrm{ST}_{\mathrm{nc}}(\perf_\dg(\cX;\cF)) \,,
\end{eqnarray*} 
where $Y_\sigma$ is a certain $\sigma^\vee$-Galois cover of $X^\sigma$ induced by the restriction of $\cF$ to $X^\sigma$. Moreover, the conjecture $\mathrm{W}^l_{\mathrm{nc}}(\perf_\dg(\cX;\cF))$ holds.
\end{theorem} 
Similarly to Theorem \ref{thm:orbifolds}, Theorem \ref{thm:twist} implies that the noncommutative Weil conjecture(s) hold(s) for every twisted global orbifold. It implies moreover that the noncommutative strong form of the Tate conjecture holds for all those twisted global orbifolds whose underlying scheme is a curve.

\subsection*{Connective dg algebras}
Let $A$ be a smooth proper {\em connective} dg $k$-algebra, i.e., $\dgHi(A)=0$ for every $i>0$. Via the Dold-Kan correspondence, it corresponds to a smooth proper simplicial $k$-algebra. The next result proves the noncommutative Weil conjecture(s) and the noncommutative strong form of the Tate conjecture for this (large) class of dg algebras: 
\begin{theorem}\label{thm:dg}
The conjectures $\mathrm{W}_{\mathrm{nc}}(A)$, $\mathrm{W}^l_{\mathrm{nc}}(A)$, and $ \mathrm{ST}_{\mathrm{nc}}(A)$, hold.
\end{theorem}
\subsection*{Finite-dimensional dg algebras}
Let $A$ be a smooth {\em finite-dimensional} dg $k$-algebra, i.e., $\mathrm{dim}_k(A^i)<\infty$ for every $i \in \mathbb{Z}$. The next result proves the noncommutative Weil conjecture(s) and the noncommutative strong form of the Tate conjecture for this (large) class of dg algebras: 
\begin{theorem}\label{thm:finite}
The conjectures $\mathrm{W}_{\mathrm{nc}}(A)$, $\mathrm{W}^l_{\mathrm{nc}}(A)$, and $ \mathrm{ST}_{\mathrm{nc}}(A)$, hold.
\end{theorem} 

\section{Applications to commutative geometry}\label{sec:applications2}
Let $k=\bbF_q$ be a finite field of characteristic $p$. As mentioned in \S\ref{sec:intro}, both the Weil conjecture as well as the strong form of the Tate conjecture hold for curves (recall that Weil proved his famous conjecture for curves using solely the classical intersection theory of divisors on surfaces). In this section, making use of Theorems \ref{thm:main} and \ref{thm:strong}, we bootstrap these results from curves to intersections of two quadrics and to linear sections of determinantal varieties. This yields an alternative noncommutative proof of the Weil conjecture for all these (higher dimensional) schemes, which avoids the involved tools used by Deligne. It yields moreover a proof of the strong form of the Tate conjecture in new cases.
\subsection*{Intersections of two quadrics}
Let $X \subset \bbP^{n-1}$ be a smooth complete intersection of two quadric hypersurfaces, with $n\geq 4$. The linear span of these two quadrics gives rise to an hypersurface $Q \subset \bbP^1 \times \bbP^{n-1}$, and the projection onto the first factor gives rise to a flat quadric fibration $f\colon Q \to \bbP^1$ of relative dimension $n-2$. 
\begin{theorem}\label{thm:quadrics}
Assume that all the fibers of $f$ have corank $\leq 1$. Under this assumption, the following holds:
\begin{itemize}
\item[(i)] When $n$ is even, the conjectures $\mathrm{W}(X)$ and $\mathrm{ST}(X)$ hold. 
\item[(ii)] When $n$ is odd and $p\neq 2$, the conjectures $\mathrm{W}(X)$ and $\mathrm{ST}(X)$ hold. 
\end{itemize}
\end{theorem}
\subsection*{Linear sections of determinantal varieties}
Let $U_1$ and $U_2$ be two finite-dimensional $k$-vector spaces of dimensions $d_1$ and $d_2$, respectively, $V:=U_1 \otimes U_2$, and $0 < r < d_1$ an integer. Consider the determinantal variety $\cZ^r_{d_1,d_2}\subset \bbP(V)$ defined as the locus of those matrices $U_2 \to U_1^\ast$ with rank $\leq r$.
\begin{example}[Segre varieties]
In the particular case where $r=1$, the determinantal varieties reduce to the classical Segre varieties. Concretely, $\cZ_{d_1, d_2}^1$ is given by the image of Segre homomorphism $\bbP(U_1) \times \bbP(U_2) \to \bbP(V)$.
\end{example}
In contrast with the Segre varieties, the varieties $\cZ^r_{d_1, d_2}$, with $r\geq 2$, are not smooth. Their singular locus consists of those matrices $U_2 \to U_1^\ast$ with rank $<r$, i.e., it agrees with the closed subvarieties $\cZ^{r-1}_{d_1, d_2}$. Nevertheless, it is well-known that $\cZ^r_{d_1, d_2}$ admits a canonical resolution of singularities $X:=\cX_{d_1, d_2}^r \to \cZ^r_{d_1, d_2}$. Dually, consider the variety $\cW^r_{d_1, d_2}\subset \bbP(V^\ast)$, defined as the locus of those matrices $U^\ast_2 \to U_1$ with corank $\geq r$, and the associated canonical resolution of singularities $Y:=\cY^r_{d_1, d_2} \to \cW^r_{d_1, d_2}$. Finally, given a linear subspace $L\subseteq V^\ast$, consider the associated linear sections $X_L:=X\times_{\bbP(V)}\bbP(L^\perp)$ and $Y_L:=Y \times_{\bbP(V^\ast)}\bbP(L)$.
\begin{theorem}\label{thm:HPD-duality}
Assume that $X_L$ and $Y_L$ are smooth\footnote{The linear section $X_L$ is smooth if and only if the linear section $Y_L$ is smooth.}, and that $\mathrm{codim}(X_L)=\mathrm{dim}(L)$ and $\mathrm{codim}(Y_L)=\mathrm{dim}(L^\perp)$. Under these assumptions (which hold for a generic choice of $L$), the following holds:
\begin{itemize}
\item[(i)] When $\mathrm{dim}(L)= r(d_1+d_2 -r)-2$, the conjectures $\mathrm{W}(Y_L)$ and $\mathrm{ST}(Y_L)$ hold.
\item[(ii)] When $\mathrm{dim}(L)= 2- r(d_1-d_2 -r)$, the conjectures $\mathrm{W}(X_L)$ and $\mathrm{ST}(X_L)$ hold.
\end{itemize}
\end{theorem}
\begin{example}[Segre varieties]\label{ex:Segre}
Let $r=1$. Thanks to Theorem \ref{thm:HPD-duality}(ii), when $\mathrm{dim}(L) = 3-d_1+d_2$, the conjectures $\mathrm{W}(X_L)$ and $\mathrm{ST}(X_L)$ hold. In all these cases, $X_L$ is a linear section of the Segre variety $\cZ^1_{d_1, d_2}$ Moreover, $\mathrm{dim}(X_L)=2d_1 -5$. Therefore, for example, by letting $d_1 \to \infty$ and by keeping $\mathrm{dim}(L)$ fixed, we obtain infinitely many examples of smooth projective $k$-schemes $X_L$, of arbitrary high dimension, satisfying the  Weil conjecture and the strong form of the Tate conjecture. In what concerns the latter conjecture, these examples are, to the best of the author's knowledge, new in the literature.
\end{example}
\begin{example}[Square matrices]\label{ex:square}
Let $d_1=d_2$. Thanks to Theorem \ref{thm:HPD-duality}(ii), when $\mathrm{dim}(L)= 2+r^2$, the conjectures $\mathrm{W}(X_L)$ and $\mathrm{ST}(X_L)$ hold. In all these cases, we have $\mathrm{dim}(X_L)=2r(d_1 - r)-3$. Therefore, for example, by letting $d_1\to \infty$ and by keeping $\mathrm{dim}(L)$ fixed, we obtain infinitely many examples of smooth projective $k$-schemes $X_L$, of arbitrary high dimension, satisfying the  Weil conjecture and the strong form of the Tate conjecture. In what concerns the latter conjecture, these examples are, to the best of the author's knowledge, new in the literature.
\end{example}
\section{Statement of results: $L$-functions}\label{sec:L-functions}
Let $X$ be a smooth proper $\bbQ$-scheme of dimension $d$. It is well-known that there exists a finite number of primes $p_1, \ldots, p_m$ and a smooth proper scheme $\mathfrak{X}$ over $\mathrm{Spec}(\bbZ[1/p_1, \ldots, 1/p_m])$ such that $X\simeq \mathfrak{X} \times_{\mathrm{Spec}(\bbZ[1/p_1, \ldots, 1/p_m])}\mathrm{Spec}(\bbQ)$. In what follows, we will assume that $X$ has {\em good reduction} at $p_1, \ldots, p_m$, i.e., we will assume that for every $i=1, \ldots, m$ there exists a smooth proper scheme $\mathfrak{X}_i$ over $\mathrm{Spec}(\bbZ_{(p_i)})$ such that $X\simeq \mathfrak{X}_i \times_{\mathrm{Spec}(\bbZ_{(p_i)})} \mathrm{Spec}(\bbQ)$. Given a prime $p\neq p_1, \ldots, p_m$, let us write $\mathfrak{X}_p := \mathfrak{X} \times_{\mathrm{Spec}(\bbZ[1/p_1, \ldots, 1/p_m])} \mathrm{Spec}(\bbF_p)$ for the fiber of $\mathfrak{X}$ at $p$. In the same vein, let us write $\mathfrak{X}_{p_i}:=\mathfrak{X}_i \times_{\mathrm{Spec}(\bbZ_{(p_i)})} \mathrm{Spec}(\bbF_{p_i})$ for the fiber of $\mathfrak{X}_i$ at $p_i$. Under these assumptions and notations, recall that the $L$-function of $X$ may be defined as the Euler product $L(X;s) := \prod_{p\neq p_1, \ldots, p_m} \zeta(\mathfrak{X}_p;s) \cdot \prod_{i=1}^m \zeta(\mathfrak{X}_{p_i};s)$.
As proved by Serre in \cite{Serre2,Serre1}, this infinite product converges (absolutely) in the half-plane $\mathrm{Re}(s)> d+1$. Moreover, $L(X;s)$ is non-zero in this half-plane region. 

In the same vein, given an integer $0 \leq w \leq 2d$, consider the associated $L$-function $
L_w(X;s) :=  \prod_{p \neq p_1, \ldots, p_m} \zeta_w(\mathfrak{X}_p;s) \cdot \prod_{i=1}^m \zeta_w(\mathfrak{X}_{p_i};s)$. Assuming the conjectures $\mathrm{W}(\mathfrak{X}_p)$ and $\mathrm{W}(\mathfrak{X}_{p_i})$ (which were proved by Deligne in the seventies), Serre proved in \cite{Serre2,Serre1} that the latter infinite product converges (absolutely) in the half-plane $\mathrm{Re}(s)>\frac{w}{2} +1$ and, moreover, that it is non-zero in this region. 

Finally, note that the weight decomposition \eqref{eq:factorization1} (applied to the fibers $\mathfrak{X}_p$ and $\mathfrak{X}_{p_i}$) yields the following weight decomposition of $L$-functions:
\begin{eqnarray}\label{eq:L-factorization}
L(X;s)= \frac{L_0(X;s)L_2(X;s) \cdots L_{2d}(X;s)}{L_1(X;s) L_3(X;s) \cdots L_{2d-1}(X;s)} && \mathrm{Re}(s)>d+1\,.
\end{eqnarray}

Let $\cA$ be a smooth proper $\bbQ$-linear dg category. Similarly to the commutative world,  there exists a finite number of primes $p_1, \ldots, p_m$ and a smooth proper $\bbZ[1/p_1, \ldots, 1/p_m]$-linear dg category $\mathfrak{A}$ such that $\cA$ and $\mathfrak{A} \otimes_{\bbZ[1/p_1, \ldots, 1/p_m]}\bbQ$ are Morita equivalent. In what follows, we will assume that $\cA$ has {\em good reduction} at $p_1, \ldots, p_m$, i.e., we will assume that for every $i=1, \ldots, m$ there exists a smooth proper $\bbZ_{(p_i)}$-linear dg category $\mathfrak{A}_i$ such that $\cA$ and $\mathfrak{A}\otimes_{\bbZ_{(p_i)}} \bbQ$ are Morita equivalent. Given a prime $p \neq p_1, \ldots, p_m$, let us write $\mathfrak{A}_p:= \mathfrak{A} \otimes^{\bf L}_{\bbZ[1/p_1, \ldots, 1/p_m]}\bbF_p$. In the same vein, let us write $\mathfrak{A}_{p_i}:=\mathfrak{A}_i\otimes^{\bf L}_{\bbZ_{(p_i)}}\bbF_{p_i}$. Under these assumptions and notations, we define the {\em even/odd $L$-function of $\cA$} as the following Euler product:
\begin{eqnarray}
L_{\mathrm{even}}(\cA;s) & := & \prod_{p \neq p_1, \ldots, p_m} \zeta_{\mathrm{even}}(\mathfrak{A}_p;s) \cdot \prod_{1\leq i \leq m} \zeta_{\mathrm{even}}(\mathfrak{A}_{p_i};s) \label{eq:L-ev}\\
L_{\mathrm{odd}}(\cA;s) &  := & \prod_{p \neq p_1, \ldots, p_m} \zeta_{\mathrm{odd}}(\mathfrak{A}_p;s) \cdot \prod_{1\leq i\leq m} \zeta_{\mathrm{odd}}(\mathfrak{A}_{p_i};s) \label{eq:L-odd}\,.
\end{eqnarray}
\begin{theorem}\label{thm:main-L-functions}
Assume that the conjectures $\mathrm{W}_{\mathrm{nc}}(\mathfrak{A}_p)$ and $\mathrm{W}_{\mathrm{nc}}(\mathfrak{A}_{p_i})$ hold. Under these assumptions, the infinite product \eqref{eq:L-ev}, resp. \eqref{eq:L-odd}, converges (absolutely) in the half-plane $\mathrm{Re}(s)>1$, resp. $\mathrm{Re}(s)>\frac{3}{2}$. Moreover, the $L$-functions $L_{\mathrm{even}}(\cA;s)$ and $L_{\mathrm{odd}}(\cA;s)$ are non-zero in these half-plane regions. 
\end{theorem}
\begin{remark}[Convergence]
Note that the $L$-function $L_{\mathrm{even}}(\cA;s)$, resp. $L_{\mathrm{odd}}(\cA;s)$, converges (absolutely) in the half-plane $\mathrm{Re}(s)>1$, resp. $\mathrm{Re}(s)>\frac{3}{2}$, if and only if the infinite product $\prod_{p \neq p_1, \ldots, p_m} \zeta_{\mathrm{even}}(\mathfrak{A}_p;s)$, resp. $\prod_{p \neq p_1, \ldots, p_m} \zeta_{\mathrm{odd}}(\mathfrak{A}_p;s)$, converges (absolutely) in the same half-plane. Therefore, Theorem \ref{thm:main-L-functions} holds similarly without the assumption that the conjectures $\mathrm{W}_{\mathrm{nc}}(\mathfrak{A}_{p_i}), 1 \leq i \leq m$, hold.
\end{remark}
\begin{remark}[Alternative proof of Serre's convergence result]
By adapting the arguments used in the proof of Theorem \ref{thm:main-L-functions} (consult \S\ref{proof:main-L-functions}), we present in \S\ref{sec:alternative-Serre} below an alternative noncommutative proof of Serre's convergence result.
\end{remark}
\begin{example}[Smooth proper schemes]\label{ex:sp}
When $\cA= \perf_\dg(X)$, with $X$ a smooth proper $\bbQ$-scheme, we can choose for $\mathfrak{A}$ the dg category $\perf_\dg(\mathfrak{X})$ and for $\mathfrak{A}_i$ the dg category $\perf_\dg(\mathfrak{X}_i)$. Hence, since the dg categories $\mathfrak{A}_p$ and $\mathfrak{A}_{p_i}$ are Morita equivalent to $\perf_\dg(\mathfrak{X}_p)$ and $\perf_\dg(\mathfrak{X}_{p_i})$, respectively, we conclude that the factorizations of Corollary \ref{cor:main} (applied to the fibers $\mathfrak{X}_p$ and $\mathfrak{X}_{p_i}$) yield the following factorizations: 
\begin{eqnarray}
L_{\mathrm{even}}(\perf_\dg(X);s) & = & \prod_{w \, \mathrm{even}} L_w(X;s+\frac{w}{2})\quad \quad \quad\,\,\, \mathrm{Re}(s)>1 \label{eq:equality-L1}\\
L_{\mathrm{odd}}(\perf_\dg(X);s) & = & \prod_{w \, \mathrm{odd}} L_w(X;s+\frac{w-1}{2}) \quad \quad \mathrm{Re}(s)>\frac{3}{2} \label{eq:equality-L2}\,.
\end{eqnarray}
Roughly speaking, this shows that the even/odd $L$-function of $\perf_\dg(X)$ may be understood as the ``weight normalization'' of the product of the $L$-functions $L_w(X;s)$.
\end{example}
\begin{example}[Riemann zeta function]
In the particular case where $X=\mathrm{Spec}(\bbQ)$, we can choose for $\mathfrak{X}$ the smooth proper scheme $\mathrm{Spec}(\bbZ)$. Consequently, we conclude from Example \ref{ex:sp} that the even $L$-function $L_{\mathrm{even}}(\perf_\dg(\mathrm{Spec}(\bbQ));s)$ agrees with the famous Riemann zeta function $L(\mathrm{Spec}(\bbQ);s) := \prod_{p} \frac{1}{1-p^{-s}}=\sum_{n \geq 1} \frac{1}{n^s}$.
\end{example}
\begin{example}[Dedekind zeta functions]\label{ex:Dedekind}
In the particular case where $X=\mathrm{Spec}(F)$, with $F$ a number field, we can choose for $\mathfrak{X}$ the smooth proper scheme $\mathrm{Spec}(\cO_F)$ (over $\mathrm{Spec}(\bbZ)$), where $\cO_F$ stands for the ring of integers of $F$. Consequently, we conclude from Example \ref{ex:sp} that the even $L$-function $L_{\mathrm{even}}(\perf_\dg(\mathrm{Spec}(\bbF));s)$ agrees with the famous Dedekind zeta function $L(\mathrm{Spec}(F);s):= \prod_p \prod_{\cP\mid p} \frac{1}{1- N(\cP)^{-s}}= \sum_{\cI \subseteq \cO_F} \frac{1}{N(\cI)^s}$, where $\cP$, resp. $\cI$, is a prime ideal, resp. ideal, of the ring of integers $\cO_K$ and $N(\cP)$, resp. $N(\cI)$, is its norm. 
\end{example}

\subsection*{Meromorphic continuation}
Let $X$ be a smooth proper $\bbQ$-scheme of dimension $d$. A classical conjecture (see Weil \cite{Weil-old}) in the theory of $L$-functions is the following:

\vspace{0.1cm}

{\it Conjecture $\mathrm{M}(X)$: The $L$-function $L(X;s)$ admits a (unique) meromorphic continuation to the entire complex plane.}

\vspace{0.1cm}

In the same vein, given an integer $0 \leq w \leq 2d$, we have the following conjecture:

\vspace{0.1cm}

{\it Conjecture $\mathrm{M}_w(X)$: The $L$-function $L_w(X;s)$ admits a (unique) meromorphic continuation to the entire complex plane.}

\vspace{0.1cm}

The latter conjecture holds for $0$-dimensional schemes, for abelian varieties with complex multiplication, for varieties of Fermat type, for certain products of modular curves, and also for certain Shimura varieties. Besides these cases (and some other cases scattered in the literature), it remains wide open. Note that thanks to the above weight decomposition \eqref{eq:L-factorization}, we have the implication $\sum_{w=0}^{2d} \mathrm{M}_w(X) \Rightarrow \mathrm{M}(X)$.
\smallbreak

Given a smooth proper $\bbQ$-linear dg category $\cA$, the above conjecture(s) admits the following noncommutative counterpart:

\vspace{0.1cm}

{\it Conjecture $\mathrm{M}_{\mathrm{nc}}(\cA)$: The even $L$-function $L_{\mathrm{even}}(\cA;s)$, resp. the odd $L$-function $L_{\mathrm{odd}}(\cA;s)$, admits a (unique) meromorphic continuation to the~entire~complex~plane.}

\begin{remark}\label{rk:key-1}
Given a smooth proper $\bbQ$-scheme $X$ of dimension $d$, note that the factorizations \eqref{eq:equality-L1}-\eqref{eq:equality-L2} yield the implication $\sum_{w=0}^{2d} \mathrm{M}_w(X) \Rightarrow \mathrm{M}_{\mathrm{nc}}(\perf_\dg(X))$. In particular, if the conjectures $\{\mathrm{M}_w(X)\}^{2d}_{w=0}$ hold, then the above factorizations \eqref{eq:equality-L1}-\eqref{eq:equality-L2} hold in the entire complex plane.
\end{remark}

\subsection*{Tate conjecture}
Let $X$ be a smooth proper $\bbQ$-scheme of dimension $d$ and $0 \leq w \leq 2d$ an even integer. In what follows, we will assume that the conjecture $\mathrm{M}_w(X)$ holds. In the mid sixties, Tate \cite{Tate} conjectured the following:

\vspace{0.1cm}

{\it Conjecture $\mathrm{T}_w(X)$: The $L$-function $L_w(X;s)$ has a unique pole at $s=\frac{w}{2}+1$.}

\vspace{0.1cm}

This conjecture holds for $0$-dimensional schemes, for certain abelian varieties with complex multiplication, for certain varieties of Fermat type, for certain products of modular curves, and also for certain Shimura varieties. Besides these cases (and some other cases scattered in the literature), it remains wide open.

\smallbreak

Let $\cA$ be a smooth proper $\bbQ$-linear dg category. In what follows, we will assume that the conjectures $\mathrm{W}_{\mathrm{nc}}(\mathfrak{A}_p)$ and $\mathrm{W}_{\mathrm{nc}}(\mathfrak{A}_{p_i})$ hold. Recall from Theorem \ref{thm:main-L-functions} that this implies that $L_{\mathrm{even}}(\cA;s)$ converges (absolutely) in the half-plane $\mathrm{Re}(s)>1$. Also, we will assume that the conjecture $\mathrm{M}_{\mathrm{nc}}(\cA)$ holds. Under these assumptions, Tate's conjecture admits the following noncommutative counterpart:

\vspace{0.1cm}

{\it Conjecture $\mathrm{T}_{\mathrm{nc}}(\cA)$: The even $L$-function $L_{\mathrm{even}}(\cA;s)$ has a unique pole at $s=1$.}

\begin{remark}\label{rk:Tate-L}
Given a smooth proper $\bbQ$-scheme $X$, note that the factorization \eqref{eq:equality-L1} yields the implication $\sum_{w\,\mathrm{even}}\mathrm{T}_w(X) \Rightarrow\mathrm{T}_{\mathrm{nc}}(\perf_\dg(X))$.
\end{remark}

\subsection*{Beilinson conjecture}
Let $X$ be a smooth proper $\bbQ$-scheme of dimension $d$ and $0 \leq w \leq 2d$ an integer. In what follows, we will assume that the conjecture $\mathrm{M}_w(X)$ holds. Given an integer $0 \leq i \leq d$, let us write $\cZ^i(X)_\bbQ/_{\!\sim \mathrm{rat}}$ for the $\bbQ$-vector space of algebraic cycles of codimension $i$ on $X$ up to rational equivalence, $\cZ^i(X)^0_\bbQ/_{\!\sim\mathrm{rat}}$ for the $\bbQ$-subspace of those algebraic cycles which are homologically trivial, and $\cZ^i(X)_\bbQ/_{\!\sim\mathrm{hom}}$ for the $\bbQ$-vector space of algebraic cycles of codimension $i$ on $X$ up to homological equivalence\footnote{Recall that in characteristic zero all the classical Weil cohomology theories are isomorphic.}. Also, given integers $i, j \in \bbZ$, let us write $H^i_{\mathrm{mot}}(X;\bbQ(j))$ for the motivic cohomology~groups~of~$X$.

In the eighties, Beilinson \cite{Beilinson2,Beilinson1, Beilinson3} conjectured the following:

\vspace{0.1cm}

{\it Conjecture $\mathrm{B}^j_{w}(X)$: The following equalities hold:
$$\mathrm{ord}_{s=j}L_w(X;s) = \begin{cases} \begin{array}{lll}
\!\!\!-\mathrm{dim}_\bbQ \cZ^{\frac{w}{2}}(X)_\bbQ/_{\!\sim\mathrm{hom}} & j=\frac{w}{2}+1 & w\,\,\mathrm{even}\\
\!\!\!\mathrm{dim}_\bbQ H_{\mathrm{mot}}^{w+1}(X;\bbQ(w+1-j)) & j\leq \frac{w}{2} & w\,\,\mathrm{even} \\
\!\!\!\mathrm{dim}_\bbQ \cZ^{\frac{w+1}{2}}(X)^0_\bbQ/_{\!\sim\mathrm{rat}} & j=\frac{w+1}{2} & w\,\,\mathrm{odd}\\
\!\!\!\mathrm{dim}_\bbQ H_{\mathrm{mot}}^{w+1}(X;\bbQ(w+1-j)) & j\leq \frac{w-1}{2} & w\,\,\mathrm{odd}\,.
\end{array}
\end{cases}$$
}

\vspace{0.1cm}

This conjecture holds for $0$-dimensional schemes, for certain elliptic curves, for certain varieties of Fermat type, for certain products of modular curves, and also for certain Shimura varieties. Besides these cases (and some other cases scattered in the literature), it remains wide open.
\begin{remark}[Tate conjecture]\label{rk:Tate}
In addition to the conjecture $\mathrm{T}_w(X)$, the conjecture $\mathrm{B}_{w}^{\frac{w}{2}+1}(X)$, with $w$ even, was also formulated by Tate in \cite{Tate}. Note the parallelism between the set of conjectures $\{\mathrm{B}_w^{\frac{w}{2}+1}(X)\}_{w\,\mathrm{even}}$ and the strong form of the Tate conjecture (consult \S\ref{sec:intro}).
\end{remark}
\begin{remark}[Birch and Swinnerton-Dyer conjecture]
In the particular case where $X$ is an elliptic curve, the Beilinson conjecture $\mathrm{B}^1_{1}(X)$ reduces to the famous Birch and Swinnerton-Dyer conjecture \cite{Wiles}, which asserts that the order $\mathrm{ord}_{s=1}L_1(X;s)$ of the $L$-function $L_1(X;s)$ at the zero $s=1$ is equal to the rank of $\mathrm{Pic}^0(X)$.
\end{remark}
Let $\cA$ be a smooth proper $\bbQ$-linear dg category. In what follows, we will assume that the conjectures $\mathrm{W}_{\mathrm{nc}}(\mathfrak{A}_p)$ and $\mathrm{W}_{\mathrm{nc}}(\mathfrak{A}_{p_i})$ hold. Recall from Theorem \ref{thm:main-L-functions} that this implies that the even $L$-function $L_{\mathrm{even}}(\cA;s)$, resp. odd $L$-function $L_{\mathrm{odd}}(\cA;s)$, converges (absolutely) in the half-plane $\mathrm{Re}(s)>1$, resp. $\mathrm{Re}(s)>\frac{3}{2}$. Also, we will assume that the conjecture $\mathrm{M}_{\mathrm{nc}}(\cA)$ holds. Recall from \S\ref{sub:homological} below the definition of the Grothendieck group $K_0(\cA)_\bbQ$, of the $\bbQ$-subspace $K_0(\cA)_\bbQ^0$ of those Grothendieck classes which are homologically trivial, and of the homological Grothendieck group $K_0(\cA)_\bbQ/_{\!\sim\mathrm{hom}}$. Under these definitions, Beilinson's conjecture admits the following noncommutative counterpart(s):

\vspace{0.1cm}

{\it Conjecture $\mathrm{B}^j_{\mathrm{nc}, \mathrm{even}}(\cA)$: The following equalities hold:
\begin{equation}\label{eq:conj1}
\mathrm{ord}_{s=j}L_{\mathrm{even}}(\cA;s) = \begin{cases} -\mathrm{dim}_\bbQ K_0(\cA)_\bbQ/_{\!\sim\mathrm{hom}} & j=1\\
\mathrm{dim}_\bbQ K_{1}(\cA)_\bbQ & j= 0\\
\mathrm{dim}_\bbQ K_{3}(\cA)_\bbQ & j= -1\,.
\end{cases}
\end{equation}

\vspace{0.1cm}

}
{\it Conjecture $\mathrm{B}^j_{\mathrm{nc}, \mathrm{odd}}(\cA)$: The following equalities hold:
\begin{equation}\label{eq:conj2}
\mathrm{ord}_{s=j}L_{\mathrm{odd}}(\cA;s) = \begin{cases} \mathrm{dim}_\bbQ K_0(\cA)^0_\bbQ & j=1\\
\mathrm{dim}_\bbQ K_{2}(\cA)_\bbQ & j=0\,.
\end{cases}
\end{equation}
}
The noncommutative Beilinson conjecture(s) was originally invisioned by Kontsevich in his seminal talks \cite{Lefschetz,IAS}. The next result relates this conjecture(s) with Beilinson's original conjecture:
\begin{theorem}\label{thm:implications}
Given a smooth proper $\bbQ$-scheme $X$, we have the implications:
$$\sum_{w\,\mathrm{even}}\mathrm{B}_{w}^{\frac{w}{2}+1}(X)  \Rightarrow \mathrm{B}_{\mathrm{nc}, \mathrm{even}}^1(\perf_\dg(X))\quad \sum_{w\,\mathrm{odd}}\mathrm{B}_{w}^{\frac{w+1}{2}}(X) \Rightarrow \mathrm{B}_{\mathrm{nc}, \mathrm{odd}}^1(\perf_\dg(X))\,.$$ 
Assuming the Beilinson-Soul\'e vanishing conjecture (i.e., $H_{\mathrm{mot}}^i(X;\bbQ(j))=0$ when $i<0$ and also when $i=0$ and $j>0$), we have moreover the implications
$$\sum_{w\,\mathrm{even}}\mathrm{B}_{w}^{\frac{w}{2}}(X) \Rightarrow \mathrm{B}_{\mathrm{nc}, \mathrm{even}}^0(\perf_\dg(X))\quad \sum_{w\,\mathrm{even}}\mathrm{B}_{w}^{\frac{w}{2}-1}(X) \Rightarrow \mathrm{B}_{\mathrm{nc}, \mathrm{even}}^{-1}(\perf_\dg(X))$$  
as well as the implication $\sum_{w\,\mathrm{odd}}\mathrm{B}_{w}^{\frac{w-1}{2}}(X) \Rightarrow \mathrm{B}_{\mathrm{nc}, \mathrm{odd}}^0(\perf_\dg(X))$.
\end{theorem}
\begin{remark}[Potential generalization]
Let $\cA$ be a smooth proper $\bbQ$-linear dg category. Motivated by the above noncommutative Beilinson conjecture(s), it is natural to ask if the following equalities should be added to \eqref{eq:conj1}-\eqref{eq:conj2}:
\begin{eqnarray}
\mathrm{ord}_{s=j}L_{\mathrm{even}}(\cA;s) = \mathrm{dim}_\bbQ K_{1-2j}(\cA)_\bbQ && j \leq -2 \label{eq:equiv111}\\
\mathrm{ord}_{s=j}L_{\mathrm{odd}}(\cA;s) = \mathrm{dim}_\bbQ K_{2-2j}(\cA)_\bbQ && j \leq -1\,. \label{eq:equiv222}
\end{eqnarray}
As explained in Remark \ref{rk:generalization2} below, when $\cA=\perf_\dg(X)$, with $X$ a smooth proper $\bbQ$-scheme, the Beilinson conjecture plus the Beilinson-Soul\'e vanishing conjecture imply that the equality \eqref{eq:equiv111}, resp. \eqref{eq:equiv222}, holds if and only if the motivic cohomology groups $\{H_{\mathrm{mot}}^{2r+2j-1}(X;\bbQ(r))\,|\,d-j+1< r \leq d-2j\}$, resp. $\{H_{\mathrm{mot}}^{2r+2j-2}(X;\bbQ(r))\,|\, d-j+1 < r \leq d-2j+1\}$, are zero. Unfortunately, to the best of the author's knowledge, nothing is known about these groups.
\end{remark}

\section{Preliminaries}\label{sec:preliminaries}
\subsection{Dg categories}\label{sub:dg}
For a survey on dg categories, we invite the reader to consult\footnote{Consult also the pioneering work \cite{Kapranov1}.} \cite{ICM-Keller}. Let $k$ be a commutative ring and $(\cC(k),\otimes, k)$ the category of (cochain) complexes of $k$-modules. A ($k$-linear) {\em dg category $\cA$} is a category enriched over $\cC(k)$ and a {\em dg functor $F\colon \cA \to \cB$} is a functor enriched over $\cC(k)$. In what follows, we will write $\dgcat(k)$ for the category of small dg categories and dg functors.

Let $\cA$ be a dg category. The opposite dg category $\cA^\op$, resp. category $\dgHo(\cA)$, has the same objects as $\cA$ and $\cA^\op(x,y):=\cA(y,x)$, resp. $\dgHo(\cA)(x,y):=H^0(\cA(x,y))$. A {\em right dg $\cA$-module} is a dg functor $M\colon \cA^\op \to \cC_\dg(k)$ with values in the dg category of complexes of $k$-modules. Let $\cC(\cA)$ be the category of right dg $\cA$-modules. Following \cite[\S3.2]{ICM-Keller}, the {\em derived category $\cD(\cA)$ of $\cA$} is defined as the localization of $\cC(\cA)$ with respect to the objectwise quasi-isomorphisms. In what follows, we will write $\cD_c(\cA)$ for the subcategory of compact objects. 

A dg functor $F\colon \cA \to \cB$ is called a {\em Morita equivalence} if it induces an equivalence between derived categories $\cD(\cA) \simeq \cD(\cB)$; see \cite[\S4.6]{ICM-Keller}. As explained in \cite[\S1.6]{book}, the category $\dgcat(k)$ admits a Quillen model structure whose weak equivalences are the Morita equivalences. Let $\Hmo(k)$ be the associated homotopy category.

The {\em tensor product $\cA\otimes_k \cB$} of dg categories is defined as follows: the set of objects of $\cA\otimes_k \cB$ is the cartesian product of the sets of objects of $\cA$ and $\cB$ and $(\cA\otimes_k \cB)((x,w), (y,z)):=\cA(x,y) \otimes_k \cB(w,z)$. As explained in \cite[\S1.1.1 and \S1.6.4]{book}, this construction gives rise to a symmetric monoidal structure $-\otimes_k -$ on the category $\dgcat(k)$, which descends $-\otimes_k^{\bf L}-$ to the homotopy category $\Hmo(k)$.

A {\em dg $\cA\text{-}\cB$-bimodule} is a dg functor $\mathrm{B}\colon \cA\otimes_k^{\bf L}\cB^\op \to \cC_\dg(k)$. An example is the dg $\cA\text{-}\cB$-bimodule $
{}_F\mathrm{B}\colon (x,w) \mapsto \cB(w, F(x))$ associated to a dg functor $F\colon\cA\to\cB$.

Following \cite{Miami,finMot,IAS,ENS}, a dg category $\cA$ is called {\em smooth} if the dg $\cA\text{-}\cA$-bimodule ${}_{\mathrm{id}}\mathrm{B}$ belongs to the subcategory $\cD_c(\cA^\op \otimes_k^{\bf L} \cA)$ and {\em proper} if all the complexes of $k$-modules $\cA(x,y)$ belong to the subcategory $\cD_c(k)$. As explained in \cite[Thm.~1.43]{book}, the smooth proper dg categories may be (conceptually) characterized as the dualizable objects of the symmetric monoidal category $\Hmo(k)$. Moreover, the dual of a smooth proper dg category $\cA$ is its opposite dg category $\cA^\op$. In what follows, we will write $\dgcat_{\mathrm{sp}}(k)$ and $\Hmo_{\mathrm{sp}}(k)$ for the full symmetric monoidal subcategories of smooth proper dg categories. 

\subsection{Noncommutative motives}\label{sub:NC}
For a book, resp. survey, on noncommutative motives, we invite the reader to consult \cite{book}, resp. \cite{survey}. Recall from \cite[\S4.1]{book} the construction of the category of {\em noncommutative Chow motives} $\NChow(k)_\bbQ$. This category is $\bbQ$-linear, additive, idempotent complete, rigid symmetric monoidal\footnote{Recall that a symmetric monoidal category is called {\em rigid} if all its objects are dualizable.}, and comes equipped with a (composed) symmetric monoidal functor:
$$U(-)_\bbQ \colon \dgcat_{\mathrm{sp}}(k) \too \Hmo_{\mathrm{sp}}(k)_\bbQ \too \NChow(k)_\bbQ\,.$$ 
Moreover, given smooth proper dg categories $\cA$ and $\cB$, we have an isomorphism:
\begin{equation}\label{eq:Homs}
\Hom_{\NChow(k)_\bbQ}(U(\cA)_\bbQ, U(\cB)_\bbQ) \simeq K_0(\cA^\op \otimes^{\bf L}_k\cB)_\bbQ\,.
\end{equation}

Recall from \cite[\S4.6]{book} the construction of the category of {\em noncommutative numerical motives} $\NNum(k)_\bbQ$. This category is also $\bbQ$-linear, additive, idempotent complete, rigid symmetric monoidal, and comes equipped with a (quotient) $\bbQ$-linear symmetric monoidal functor $\NChow(k)_\bbQ \to \NNum(k)_\bbQ$.

\subsection{Numerical Grothendieck group}\label{sub:numerical}
Let $k$ be a field. Given a smooth proper $k$-linear dg category $\cA$, recall that its Grothendieck group $K_0(\cA):=K_0(\cD_c(\cA))$ comes equipped with the Euler bilinear pairing $\chi\colon K_0(\cA) \times K_0(\cA) \to \bbZ$ defined as follows $([M],[N])\mapsto \sum_n (-1)^n \mathrm{dim}_k\Hom_{\cD_c(\cA)}(M,N[n])$.
This pairing is not symmetric neither skew-symmetric. Nevertheless, making use of the notion of Serre functor developed in \cite{Kapranov2}, it can be shown that the associated left and right kernels of $\chi$ agree; consult \cite[Prop.~4.24]{book}. Hence, the  {\em numerical Grothendieck group} $K_0(\cA)/_{\!\sim \mathrm{num}}$ is defined as the quotient of $K_0(\cA)$ by the kernel of $\chi$. As proved in \cite[Thm.~5.1]{finite}, $K_0(\cA)/_{\!\sim \mathrm{num}}$ is a finitely generated free abelian group. In what follows, we will write $K_0(\cA)_\bbQ/_{\!\sim \mathrm{num}}$ for the associated finite-dimensional $\bbQ$-vector space $K_0(\cA)/_{\!\sim \mathrm{num}} \otimes_\bbZ\bbQ$. Finally, recall from \cite[\S4.6-\S4.7]{book} that, given smooth proper dg categories $\cA$ and $\cB$, we have an isomorphism:
\begin{equation}\label{eq:numerical}
\Hom_{\NNum(k)_\bbQ}(U(\cA)_\bbQ, U(\cB)_\bbQ) \simeq K_0(\cA^\op\otimes_k \cB)_\bbQ/_{\!\sim \mathrm{num}}\,.
\end{equation}
\subsection{Homological Grothendieck group}\label{sub:homological}
Let $k$ be a field of characteristic zero. Following \cite[\S9]{JEMS}, periodic cyclic homology gives rise to a $\bbQ$-linear functor
\begin{equation}\label{eq:HP}
HP_\ast(-)\colon \NChow(k)_\bbQ \too \mathrm{mod}_{\bbZ}(k[v^{\pm1}])
\end{equation}
with values in the category of (degreewise finite-dimensional) $\bbZ$-graded $k[v^{\pm1}]$-modules, where $v$ is a variable of degree $-2$. Therefore, given a smooth proper $k$-linear dg category $\cA$, by combining the functor \eqref{eq:HP} with the above identification \eqref{eq:Homs} (with $\cA:=k$ and $\cB:=\cA$), we obtain an induced $\bbQ$-linear homomorphism $\mathrm{ch}\colon K_0(\cA)_\bbQ \to HP_0(\cA)$. Under these notations, the {\em homological Grothendieck group $K_0(\cA)_\bbQ/_{\!\sim \mathrm{hom}}$} is defined as the quotient of $K_0(\cA)_\bbQ$ by the kernel of $\mathrm{ch}$. In the same vein, $K_0(\cA)_\bbQ^0$ is defined as the kernel of $\mathrm{ch}$.
\section{Topological periodic cyclic homology}\label{sec:cyclotomic}
Let $k=\bbF_q$ be a finite field of characteristic $p$, with $q=p^r$, $W(k)$ the ring of $p$-typical Witt vectors of $k$, and $K:=W(k)_{1/p}$ the fraction field of $W(k)$. For recent/modern references on topological (periodic) cyclic homology, we invite the reader to consult \cite{Lars,NS} (and also \cite{Ahkil, BM}). Topological periodic cyclic homology gives rise to a symmetric monoidal functor $TP_\ast(-)_{1/p}\colon \dgcat_{\mathrm{sp}}(k) \to \mathrm{mod}_\bbZ(K[v^{\pm1}])$ with values in the category of (degreewise finite-dimensional) $\bbZ$-graded $K[v^{\pm1}]$-modules, where $v$ is a variable of degree $-2$. As explained in \cite[Thm.~2.3]{finite}, this functor yields a $\bbQ$-linear symmetric monoidal functor:
\begin{equation}\label{eq:TP2}
TP_\ast(-)_{1/p}\colon \NChow(k)_\bbQ \too \mathrm{mod}_\bbZ(K[v^{\pm1}])\,.
\end{equation}
\subsection{Cyclotomic Frobenius}\label{sub:cyclotomic}
Let $\cA$ be a smooth proper $k$-linear dg category. By construction, its topological Hochschild homology $THH(\cA)$ carries a canonical cyclotomic structure in the sense of \cite[\S2]{NS}. Using the $S^1$-action on $THH(\cA)$, we can consider the spectrum of homotopy orbits $THH(\cA)_{hS^1}$, the spectrum of homotopy fixed-points $TC^-(\cA):=THH(\cA)^{h S^1}$, and also the Tate construction $TP(\cA):=THH(\cA)^{t S^1}$ in the sense of Greenlees \cite{Greenlees}. As explained in \cite[Cor.~I.4.3]{NS}, these spectra are related by the following cofiber sequence
\begin{equation}\label{eq:cofiber}
\Sigma \,THH(\cA)_{hS^1} \stackrel{N}{\too} THH(\cA)^{hS^1} \stackrel{\mathrm{can}}{\too} THH(\cA)^{tS^1}\,, 
\end{equation}
where $N$ stands for the norm map. It is well-known that the abelian groups $THH_\ast(\cA)$ are $k$-linear. Hence, after inverting $p$, we have $\Sigma \,THH(\cA)_{hS^1}[1/p]\simeq \ast$. Consequently, the above cofiber sequence \eqref{eq:cofiber} leads to a canonical isomorphism:
\begin{equation}\label{eq:canonical}
\mathrm{can}\colon TC_\ast^-(\cA)_{1/p} \stackrel{\simeq}{\too} TP_\ast(\cA)_{1/p}\,.
\end{equation}
It is also well-known that the spectrum $THH(\cA)$ is a dualizable $THH(k)$-module spectrum. Thanks to B\"okstedt's celebrated computation $THH_\ast(k)\simeq k[u]$, where $u$ is a variable of degree $2$, this implies that the spectrum $THH(\cA)$ is {\em bounded below}, i.e., there exists an integer $m \in \bbZ$ such that $THH_n(\cA)=0$ for every $n <m$. Since the abelian groups $THH_\ast(\cA)$ are $k$-linear, this also implies that the spectrum $THH(\cA)$ is $p$-complete. Therefore, as explained in \cite[Lem.~II.4.2]{NS}, the cyclotomic structure of $THH(\cA)$ yields another homomorphism:
\begin{equation}\label{eq:cyclotomic}
\varphi_p\colon TC^-_\ast(\cA)_{1/p} \too TP_\ast(\cA)_{1/p}\,.
\end{equation} 
It follows from \cite[Prop.~4.7]{Ahkil} that the homomorphism  \eqref{eq:cyclotomic} is invertible. Hence, let us write $\varphi:=\varphi_p \circ \mathrm{can}^{-1}$ for the induced automorphism of $TP_\ast(\cA)_{1/p}$. The automorphism $\varphi$ is not $K$-linear. Instead, it is semilinear with respect to the isomorphism $K\stackrel{\simeq}{\to} K$ induced by the Frobenius map $\lambda \mapsto \lambda^p$ on $k$. Therefore, its $r$-fold composition $\mathrm{F}_\ast:=\varphi^r$ becomes a $K$-linear automorphism of $TP_\ast(\cA)_{1/p}$.\
\begin{notation}\label{not:Frobenius}
The $K$-linear automorphism $\mathrm{F}_\ast$ is called the {\em cyclotomic Frobenius}.
\end{notation}
The cyclotomic Frobenius $\mathrm{F}_\ast$ is not $K[v^{\pm1}]$-linear. Instead, it is semilinear with respect to the $\bbZ$-graded $K$-algebra isomorphism $\tau\colon K[v^{\pm1}] \stackrel{\simeq}{\to} K[v^{\pm1}], v \mapsto qv$. In other words, we have the following commutative squares (for every $n \in \bbZ$):
\begin{equation}\label{eq:squares}
\xymatrix{
TP_n(\cA)_{1/p} \ar[rr]^-{v\cdot -}_-\simeq \ar[d]_-{\mathrm{F}_n}^-\simeq && TP_{n-2}(\cA)_{1/p} \ar[d]^-{\frac{1}{q}\cdot \mathrm{F}_{n-2}}_-\simeq \\
TP_n(\cA)_{1/p} \ar[rr]^-{v\cdot -}_-\simeq && TP_{n-2}(\cA)_{1/p}\,.
}
\end{equation}
Consequently, we obtain an induced isomorphism of $\bbZ$-graded $K[v^{\pm1}]$-modules:
$$ \mathrm{F}_\ast\colon TP_\ast(\cA)_{1/p}^\tau:= TP_\ast(\cA)_{1/p} \otimes_{K[v^{\pm1}], \tau} K[v^{\pm1}] \stackrel{\simeq}{\too} TP_\ast(\cA)_{1/p}\,.$$
Given smooth proper $k$-linear dg categories $\cA$ and $\cB$, we have a natural isomorphism $\mathrm{F}_\ast^{\cA\otimes_k \cB} \simeq \mathrm{F}_\ast^\cA \otimes_{K[v^{\pm1}]} \mathrm{F}_\ast^\cB$. Therefore, by construction of the category $\NChow(k)_\bbQ$, the assignment $U(\cA)_\bbQ \mapsto \mathrm{F}_\ast^\cA$ (parametrized by the smooth proper dg categories $\cA$) yields a $\bbQ$-linear symmetric monoidal natural transformation from the functor
\begin{equation}\label{eq:twist}
TP_\ast(-)_{1/p}^\tau\colon \NChow(k)_\bbQ \too \mathrm{mod}_\bbZ(K[v^{\pm1}])
\end{equation}
to the above functor \eqref{eq:TP2}.
\begin{remark}[Loss of information]\label{rk:semilinearity}
The above commutative squares \eqref{eq:squares} show that there is no loss of information in working solely with $\mathrm{F}_0$ and $\mathrm{F}_1$ (as~done~as~\S\ref{sec:intro}).

Similarly to \S\ref{sec:intro}, given an integer $n \in \bbZ$ and an embedding $\iota\colon K \hookrightarrow \bbC$, we can consider the following Hasse-Weil zeta function:
$$ \zeta_n(\cA;s):=\mathrm{det}(\id - q^{-s}(\mathrm{F}_n\otimes_{K, \iota}\bbC)|TP_n(\cA)_{1/p}\otimes_{K, \iota}\bbC)^{-1}\,.$$
Thanks to the above squares \eqref{eq:squares}, we have $\zeta_n(\cA;s)= \zeta_0(\cA; s+ \frac{n}{2})$ when $n$ is even and $\zeta_n(\cA;s)=\zeta_1(\cA; s+ \frac{n-1}{2})$ when $n$ is odd. Consequently, there is no loss of information in working solely with the even/odd Hasse-Weil zeta functions $\zeta_{\mathrm{even}}(\cA;s):= \zeta_0(\cA;s)$ and $\zeta_{\mathrm{odd}}(\cA;s):=\zeta_1(\cA;s)$ (as done in \S\ref{sec:intro}).
\end{remark}
\subsection{Enriched topological periodic cyclic homology functor}\label{sub:enriched}
Let us write $\mathrm{Aut}(K[v^{\pm1}])^\tau$ for the category of $\tau$-semilinear automorphisms in $\mathrm{mod}_\bbZ(K[v^{\pm1}])$. Recall that an object of $\mathrm{Aut}(K[v^{\pm1}])^\tau$ is a pair $(V_\ast, f_\ast)$, where $V_\ast$ is a (degreewise finite-dimensional) $\bbZ$-graded $K[v^{\pm1}]$-module $V_\ast$ and $f_\ast\colon V_\ast^\tau \stackrel{\simeq}{\to} V_\ast$ is an automorphism. Thanks to the considerations of \S\ref{sub:cyclotomic}, note that by combining the functor \eqref{eq:TP2} with the cyclotomic Frobenius natural transformation, we obtain the following $\bbQ$-linear symmetric monoidal functor:
\begin{eqnarray}\label{eq:enriched}
\NChow(k)_\bbQ \too \mathrm{Aut}(K[v^{\pm1}])^\tau && U(\cA)_\bbQ \mapsto (TP_\ast(\cA)_{1/p}, \mathrm{F}_\ast)\,.
\end{eqnarray}
In what follows, we will call \eqref{eq:enriched} the {\em enriched topological periodic cyclic homology functor}. As explained in \S\ref{sub:embedding} below, this functor enables an alternative formulation of the noncommutative strong form of the Tate conjecture (when all smooth proper dg categories are considered simultaneously).
\section{Proof of Theorem \ref{thm:main}}
Following \cite[Thm.~2]{Elment}\cite[Thm.~5.2]{finite-CD} (this is a result of Scholze), we have natural isomorphisms of finite-dimensional $K$-vector spaces:
\begin{eqnarray}\label{eq:Scholze}
TP_0(\perf_\dg(X))_{1/p} & \simeq & \bigoplus_{w\,\mathrm{even}}H^w_{\mathrm{crys}}(X) \label{eq:Scholze1}\\
TP_1(\perf_\dg(X))_{1/p} & \simeq &  \bigoplus_{w\,\mathrm{odd}}H^w_{\mathrm{crys}}(X)\,. \label{eq:Scholze2}
\end{eqnarray}
Moreover, following \cite[\S7]{Lars}, the cyclotomic Frobenius $\mathrm{F}_0$ corresponds under the above isomorphism \eqref{eq:Scholze1} to the following automorphism:
\begin{equation}\label{eq:auto-1}
\bigoplus_{w\,\mathrm{even}}q^{-\frac{w}{2}}\mathrm{Fr}^w\colon \bigoplus_{w \, \mathrm{even}} H^w_{\mathrm{crys}}(X) \stackrel{\simeq}{\too} \bigoplus_{w\,\mathrm{even}} H^w_{\mathrm{crys}}(X)\,.
\end{equation}
Similarly, the cyclotomic Frobenius $\mathrm{F}_1$ corresponds under the above isomorphism \eqref{eq:Scholze2} to the following automorphism:
\begin{equation}\label{eq:auto-2}
\bigoplus_{w\,\mathrm{odd}}q^{-\frac{w-1}{2}}\mathrm{Fr}^w\colon \bigoplus_{w \, \mathrm{odd}} H^w_{\mathrm{crys}}(X) \stackrel{\simeq}{\too} \bigoplus_{w\,\mathrm{odd}} H^w_{\mathrm{crys}}(X)\,.
\end{equation}
As usual, let $d:=\mathrm{dim}(X)$. Also, given an integer $0 \leq w\leq 2d$, let $\beta_w:=\mathrm{dim}_KH^w_{\mathrm{crys}}(X)$ and $\{\lambda_{(w,1)}, \ldots, \lambda_{(w,\beta_w)}\}$ be the eigenvalues (with multiplicity) of the automorphism $\mathrm{Fr}^w$. Thanks to \eqref{eq:auto-1}-\eqref{eq:auto-2}, the eigenvalues of the cyclotomic Frobenius $\mathrm{F}_0$ are given by $\bigcup_{w\,\mathrm{even}}\{q^{-\frac{w}{2}}\lambda_{(w,1)}, \ldots, q^{-\frac{w}{2}}\lambda_{(w,\beta_w)}\}$ and the eigenvalues of the cyclotomic Frobenius $\mathrm{F}_1$ are given by $\bigcup_{w\,\mathrm{odd}}\{q^{-\frac{w-1}{2}}\lambda_{(w,1)}, \ldots, q^{-\frac{w-1}{2}}\lambda_{(w,\beta_w)}\}$. Consequently, since the numbers $q^{-\frac{w}{2}}$, with $w$ even, and $q^{-\frac{w-1}{2}}$, with $w$ odd, are rational, we conclude that $\mathrm{W}_{\mathrm{nc}}(\perf_\dg(X))$ holds if and only if $\mathrm{W}(X)$ holds. 
\subsection*{Proof of Corollary \ref{cor:main}}
Thanks to \eqref{eq:auto-1}-\eqref{eq:auto-2}, note that we have the following description of the even/odd Hasse-Weil zeta function:
\begin{eqnarray*}
\zeta_{\mathrm{even}}(\perf_\dg(X); s) & = & \prod_{w\, \mathrm{even}}\mathrm{det}(\id - q^{-(s + \frac{w}{2})}(\mathrm{Fr}^w \otimes_{K, \iota}\bbC)|H^w_{\mathrm{crys}}(X)\otimes_{K, \iota}\bbC)^{-1} \\
\zeta_{\mathrm{odd}}(\perf_\dg(X); s) & = & \prod_{w\, \mathrm{odd}}\mathrm{det}(\id - q^{-(s + \frac{w-1}{2})}(\mathrm{Fr}^w \otimes_{K, \iota}\bbC)|H^w_{\mathrm{crys}}(X)\otimes_{K, \iota}\bbC)^{-1}\,.
\end{eqnarray*}
Since the polynomials $\mathrm{det}(\id - t \mathrm{Fr}^w |H^w_{\mathrm{crys}}(X))$ have integer coefficients, we hence conclude that the right-hand sides of the above equalities are equal to the products $\prod_{w\,\mathrm{even}}\zeta_w(X;s+\frac{w}{2})$ and $\prod_{w\,\mathrm{odd}}\zeta_w(X;s+\frac{w-1}{2})$, respectively.

\section{Proof of Theorem \ref{thm:new}}

Since the proof uses some results of \S\ref{sec:functional}, we postpone it to \S\ref{sec:new}.

\section{Proof of Theorem \ref{thm:functional}}\label{sec:functional}
We start by recalling the following general result, whose proof is a simple linear algebra exercise that we leave for the reader.
\begin{lemma}\label{lem:aux2}
Let $\theta\colon V\otimes_K W \to K$ a perfect bilinear pairing of finite-dimensional $K$-vector spaces, $f$ an automorphism of $V$, $g$ an automorphism of $W$, and $\lambda \in K$ a non-zero scalar, making the following diagram commute:
$$
\xymatrix@C=2.2em@R=3.2em{
V\otimes_K W \ar[d]_-{f\otimes_K g}^-\simeq \ar[r]^-{\theta} & K \ar[d]^-{\lambda\cdot -}_-\simeq \\
V\otimes_K W \ar[r]_-{\theta}& K\,.
}
$$
Under these assumptions, we have the following equality of polynomials:
$$ \mathrm{det}(\id - t g|W) = \frac{(-1)^{\mathrm{dim}(V)} \lambda^{\mathrm{dim}(V)} t^{\mathrm{dim}(V)}}{\mathrm{det}(f|V)}\cdot \mathrm{det}(\id - \lambda^{-1}t^{-1} f|V)\,.$$ 
\end{lemma}

\begin{proposition}\label{prop:pairings}
Given a smooth proper $k$-linear dg category $\cA$, there exist perfect bilinear pairings $\theta_0$ and $\theta_1$ making the following diagrams commute:
$$
\xymatrix@C=2.2em@R=3.2em{
TP_0(\cA^\op)_{1/p} \otimes_K TP_0(\cA)_{1/p} \ar[d]_-{\mathrm{F}_0 \otimes_K \mathrm{F}_0}^{\simeq} \ar[r]^-{\theta_0} & K \ar@{=}[d] & TP_1(\cA^\op)_{1/p} \otimes_K TP_1(\cA)_{1/p} \ar[d]_-{\mathrm{F}_1 \otimes_K \mathrm{F}_1}^{\simeq} \ar[r]^-{\theta_1} & K \ar[d]^-{q\cdot -}_-{\simeq} \\
TP_0(\cA^\op)_{1/p} \otimes_K TP_0(\cA)_{1/p} \ar[r]_-{\theta_0} & K & TP_1(\cA^\op)_{1/p} \otimes_K TP_1(\cA)_{1/p} \ar[r]_-{\theta_1} & K\,.
}
$$
\end{proposition}
\begin{proof}
Recall from \S\ref{sub:cyclotomic} that the assignment $U(\cA)_\bbQ \mapsto \mathrm{F}^\cA_\ast$ (parametrized by the smooth proper dg categories $\cA$) yields a $\bbQ$-linear symmetric monoidal natural transformation from the functor \eqref{eq:twist} to the functor \eqref{eq:TP2}. Recall also from \S\ref{sub:dg}-\S\ref{sub:NC} that $U(\cA)_\bbQ$ is a dualizable object of the symmetric monoidal category $\NChow(k)_\bbQ$ and that $U(\cA^\op)_\bbQ$ is the dual of $U(\cA)_\bbQ$. Consequently, by applying the functors \eqref{eq:twist} and \eqref{eq:TP2} to the evaluation morphism $U(\cA^\op)_\bbQ\otimes U(\cA)_\bbQ \to U(k)_\bbQ$, we obtain the following commutative diagram:
\begin{equation}\label{eq:square-aux}
\xymatrix@C=2.2em@R=3.2em{
TP_\ast(\cA^\op)_{1/p}^\tau \otimes_{K[v^{\pm1}]}TP_\ast(\cA)_{1/p}^\tau \ar[d]_-{\mathrm{F}_\ast\otimes_{K[v^{\pm1}]}\mathrm{F}_\ast}^-\simeq \ar[rr] && TP_\ast(k)_{1/p}^\tau=K[v^{\pm1}]^\tau \ar[d]^-{\mathrm{F}_\ast^k}_-\simeq \\
TP_\ast(\cA^\op)_{1/p} \otimes_{K[v^{\pm1}]} TP_\ast(\cA)_{1/p} \ar[rr] && TP_\ast(k)_{1/p} = K[v^{\pm1}]\,.
}
\end{equation}
Note that $TP_n(-)_{1/p}^\tau=TP_n(-)_{1/p}$ for every (fixed) integer $n \in \bbZ$. Hence, we define the left-hand side commutative diagram of Proposition \ref{prop:pairings} as the following composition
$$
\xymatrix@C=1em@R=3.2em{
TP_0(\cA^\op)_{1/p}^\tau \otimes_K TP_0(\cA)_{1/p}^\tau \ar[d]_-{\mathrm{F}_0 \otimes_K\mathrm{F}_0}^-\simeq \ar[r] & (TP_\ast(\cA^\op)_{1/p}^\tau\otimes_{K[v^{\pm1}]}TP_\ast(\cA)_{1/p}^\tau)_0 \ar[d]_-{(\mathrm{F}_\ast\otimes_{K[v^{\pm1}]}\mathrm{F}_\ast)_0}^-\simeq \ar[r] & TP_0(k)_{1/p}^\tau \ar@{=}[d]^-{\mathrm{F}_0=\id} \\
TP_0(\cA^\op)_{1/p} \otimes_K TP_0(\cA)_{1/p} \ar[r] & (TP_\ast(\cA^\op)_{1/p}\otimes_{K[v^{\pm1}]}TP_\ast(\cA)_{1/p})_0 \ar[r] & TP_0(k)_{1/p}\,,
}
$$
where the left-hand side horizontal morphisms are induced by the monoidal structure of the category $\mathrm{mod}_\bbZ(K[v^{\pm1}])$ and the right-hand side horizontal morphisms are induced from \eqref{eq:square-aux}. By construction, the horizontal composition(s), denoted by $\theta_0$, is a perfect bilinear pairing. Similarly, the right-hand side commutative diagram of Proposition \ref{prop:pairings} is defined as the composition
$$
\xymatrix@C=1em@R=3.2em{
TP_1(\cA^\op)_{1/p}^\tau \otimes_K TP_1(\cA)_{1/p}^\tau \ar[d]_-{\mathrm{F}_1\otimes_K\mathrm{F}_1}^-\simeq \ar[r] & (TP_\ast(\cA^\op)_{1/p}^\tau\otimes_{K[v^{\pm1}]}TP_\ast(\cA)_{1/p}^\tau)_2 \ar[d]_-{(\mathrm{F}_\ast\otimes_{K[v^{\pm1}]}\mathrm{F}_\ast)_2}^-\simeq \ar[r] & TP_2(k)_{1/p}^\tau \ar[d]_-\simeq^-{\mathrm{F}_2=q\cdot -} \\
TP_1(\cA^\op)_{1/p} \otimes_K TP_1(\cA)_{1/p} \ar[r] & (TP_\ast(\cA^\op)_{1/p}\otimes_{K[v^{\pm1}]}TP_\ast(\cA)_{1/p})_2 \ar[r] & TP_2(k)_{1/p}\,,
}
$$
where the left-hand side horizontal morphisms are induced by the monoidal structure of the category $\mathrm{mod}_\bbZ(K[v^{\pm1}])$ and the right-hand side horizontal morphisms are induced from \eqref{eq:square-aux}. By construction, the horizontal composition(s), denoted by $\theta_1$, is a perfect bilinear pairing. 
\end{proof}
By construction of periodic cyclic homology, we have $TP_\ast(\cA^\op)_{1/p}=TP_\ast(\cA)_{1/p}$ (as $\mathbb{Z}$-graded $K[v^{\pm}]$-modules) and $\mathrm{F}_\ast^{\cA^\op}=\mathrm{F}_\ast^{\cA}$. Therefore, by applying the above general Lemma \ref{lem:aux2} to the perfect bilinear pairings $\theta_0$ and $\theta_1$ of Proposition \ref{prop:pairings}, we hence obtain the following equalities of polynomials:
$$
\mathrm{det}(\id - t\mathrm{F}_0|TP_0(\cA)_{1/p}) = \frac{(-1)^{\chi_0(\cA)}t^{\chi_0(\cA)}}{\mathrm{det}(\mathrm{F}_0|TP_0(\cA)_{1/p})}\cdot \mathrm{det}(\id - t^{-1}\mathrm{F}_0|TP_0(\cA)_{1/p})
$$
$$
\mathrm{det}(\id - t\mathrm{F}_1^\cA|TP_1(\cA)_{1/p}) = \frac{(-1)^{\chi_1(\cA)}q^{\chi_1(\cA)}t^{\chi_1(\cA)}}{\mathrm{det}(\mathrm{F}_1|TP_1(\cA)_{1/p})} \cdot \mathrm{det}(\id - q^{-1}t^{-1}\mathrm{F}_1|TP_1(\cA)_{1/p})\,.
$$
Now, choose an embedding $\iota\colon K \hookrightarrow \bbC$ and replace $\mathrm{F}_0$ and $\mathrm{F}_1$ by $\mathrm{F}_0\otimes_{K, \iota} \bbC$ and $\mathrm{F}_1\otimes_{K, \iota} \bbC$, respectively. By further replacing $t$ by $q^{-s}$, and then by passing to the inverse, we hence obtain the sought functional equations of Theorem \ref{thm:functional}.
\subsection*{Proof of Corollary \ref{cor:functional}}
Note that the isomorphisms \eqref{eq:Scholze1}-\eqref{eq:Scholze2} imply that $\chi_0(\perf_\dg(X))=\chi_{\mathrm{even}}(X)$ and $\chi_1(\perf_\dg(X))=\chi_{\mathrm{odd}}(X)$. Note also that the descriptions \eqref{eq:auto-1} and \eqref{eq:auto-2} of $\mathrm{F}_0$ and $\mathrm{F}_1$, respectively, lead to the following equalities
\begin{eqnarray}
\mathrm{det}(\mathrm{F}_0\otimes_{K, \iota}\bbC) & = &  \prod_{w\,\mathrm{even}}q^{-\frac{w}{2}\beta_w} \cdot \mathrm{det}(\mathrm{Fr}^w \otimes_{K, \iota}\bbC) \label{eq:equiv11} \\
\mathrm{det}(\mathrm{F}_1\otimes_{K, \iota}\bbC) & = & \prod_{w\,\mathrm{odd}}q^{-\frac{w-1}{2}\beta_w} \cdot \mathrm{det}(\mathrm{Fr}^w \otimes_{K, \iota}\bbC) \label{eq:equiv22}\,,
\end{eqnarray}
where $\beta_w:=\mathrm{dim}_KH^w_{\mathrm{crys}}(X)$. Now, recall, for example from \cite[App.~C Thm.~4.4]{Hartshorne}, that we have the following equalities
\begin{eqnarray}\label{eq:equalities-last}
\mathrm{det}(\mathrm{Fr}^{2d-w}\otimes_{K, \iota}\bbC)= \frac{q^{d\beta_w}}{\mathrm{det}(\mathrm{Fr}^w \otimes_{K, \iota}\bbC)} && 0 \leq w \leq 2d\,,
\end{eqnarray}
where $d:=\mathrm{dim}(X)$. By combining the equalities \eqref{eq:equalities-last} with the fact that $\beta_w=\beta_{2d-w}$ for every $0 \leq w \leq 2d$, we hence conclude (via a simple computation) that the square of \eqref{eq:equiv11}, resp. \eqref{eq:equiv22}, is equal to $1$, resp. to $q^{\chi_{\mathrm{odd}}(X)}$. These considerations imply that the functional equations of Theorem \ref{thm:functional}, with $\cA=\perf_\dg(X)$, reduce to the following functional equations:
\begin{eqnarray*}
\zeta_{\mathrm{even}}(\perf_\dg(X);s) & = & \pm q^{\chi_{\mathrm{even}}(X)s} \cdot \zeta_{\mathrm{even}}(\perf_\dg(X); -s)\\
\zeta_{\mathrm{odd}}(\perf_\dg(X);s)& = & \pm q^{\chi_{\mathrm{odd}}(X)s} \cdot \zeta_{\mathrm{odd}}(\perf_\dg(X); 1-s)\,.
\end{eqnarray*}
Consequently, the proof follows now from Corollary \ref{cor:main}.

\section{Proof of Theorem \ref{thm:new}}\label{sec:new}
Let $\lambda$ be an eigenvalue of the automorphism $\mathrm{F}_0$. If, by hypothesis, $q^C\lambda$ is an algebraic integer, then $\lambda$ is, in particular, an algebraic number. Therefore, it suffices to prove that all the $l$-adic conjugates of $\lambda$ have absolute value $1$.

As explained in \S\ref{sec:functional}, there exists a perfect bilinear pairing $\theta_0$ making the following diagram commute (consult Proposition \ref{prop:pairings} and the subsequent arguments):
\begin{equation}\label{eq:pairing-new}
\xymatrix{
TP_0(\cA)_{1/p} \otimes_K TP_0(\cA)_{1/p} \ar[d]_-{\mathrm{F}_0 \otimes_K \mathrm{F}_0}^{\simeq} \ar[rr]^-{\theta_0} && K \ar@{=}[d] \\
TP_0(\cA)_{1/p} \otimes_K TP_0(\cA)_{1/p} \ar[rr]_-{\theta_0} && K \,.
}
\end{equation}
Thanks to Lemma \ref{lem:aux2}, this implies that whenever $\lambda$ is an eigenvalue of the automorphism $\mathrm{F}_0$, $\frac{1}{\lambda}$ is also an eigenvalue of the automorphism $\mathrm{F}_0$.

Let us write $\bbQ(\lambda)/\bbQ$ for the (finite) field extension of $\bbQ$ generated by $\lambda$, $\cO \subset \bbQ(\lambda)$ for the associated ring of integers, and $(q^C\lambda)=\mathfrak{P}_1\cdots \mathfrak{P}_n$ and $(q^C\frac{1}{\lambda})=\mathfrak{P}'_1\cdots \mathfrak{P}'_{m}$ for the (unique) prime decomposition in $\cO$ of the ideals generated by the algebraic integers $q^C\lambda$ and $q^C \frac{1}{\lambda}$, respectively. Since $\cO$ is a Dedekind domain, we have the following (unique) prime decomposition:
$$ (q^{2C})= (q^C\lambda)(q^C \frac{1}{\lambda})= \mathfrak{P}_1\cdots \mathfrak{P}_n \mathfrak{P}'_1 \cdots \mathfrak{P}'_{m}\,.$$
This implies that all the prime ideals $\mathfrak{P}_1, \ldots, \mathfrak{P}_n, \mathfrak{P}'_1, \ldots, \mathfrak{P}'_{m}$ lie over $p \in \bbZ$. Consequently, since $\lambda=(q^C)^{-1}(q^C\lambda)$, we conclude that the absolute value of all the $l$-adic conjugates of $\lambda$ is necessarily equal to $1$. This finishes the proof in the case of the eigenvalues of $\mathrm{F}_0$. The proof in the case of the eigenvalues of $\mathrm{F}_1$ is similar. Simply replace \eqref{eq:pairing-new} by the commutative diagram:
\begin{equation}\label{eq:pairing-new1}
\xymatrix{
TP_1(\cA)_{1/p} \otimes_K TP_1(\cA)_{1/p} \ar[d]_-{\mathrm{F}_1 \otimes_K \mathrm{F}_1}^{\simeq} \ar[rr]^-{\theta_1} && K \ar[d]^-{q\cdot -}_-{\simeq} \\
TP_1(\cA)_{1/p} \otimes_K TP_1(\cA)_{1/p} \ar[rr]_-{\theta_1} && K \,.
}
\end{equation}
Similarly to \eqref{eq:pairing-new}, \eqref{eq:pairing-new1} implies that whenever $\lambda$ is an eigenvalue of the automorphism $\mathrm{F}_1$, $\frac{q}{\lambda}$ is also an eigenvalue of the automorphism $\mathrm{F}_1$.
\begin{remark}[Smooth proper schemes]\label{rk:new}
Let $X$ be a smooth proper $k$-scheme of dimension $d$. Given an integer $0 \leq w\leq 2d$, let $\beta_w:=\mathrm{dim}_KH_{\mathrm{crys}}^w(X)$ and $\{\lambda_{(w,1)}, \cdots, \lambda_{(w, \beta_w)}\}$ be the eigenvalues (with multiplicities) of the automorphism $\mathrm{Fr}^w$. As explained in the proof of Theorem \ref{thm:main}, the eigenvalues of the cyclotomic Frobenius $\mathrm{F}_0$ ($\mathrm{F}_0$ is an automorphism of $TP_0(\perf_\dg(X))_{1/p}$) are given by $\bigcup_{w\,\mathrm{even}}\{q^{-\frac{w}{2}}\lambda_{(w,1)}, \ldots, q^{-\frac{w}{2}}\lambda_{(w,\beta_w)}\}$ and the eigenvalues of the cyclotomic Frobenius $\mathrm{F}_1$ are given by $\bigcup_{w\,\mathrm{odd}}\{q^{-\frac{w-1}{2}}\lambda_{(w,1)}, \ldots, q^{-\frac{w-1}{2}}\lambda_{(w,\beta_w)}\}$. It is well-known that the eigenvalues $\lambda_{(w,1)}, \ldots, \lambda_{(w, \beta_w)}$ are algebraic integers. Therefore, by taking $C:=d$, we conclude that the eigenvalues of the automorphisms $\mathrm{F}_0$ and $\mathrm{F}_1$ become algebraic integers after multiplication by $q^C$. In other words, the assumption of Theorem \ref{thm:new} holds when $\cA=\perf_\dg(X)$.
\end{remark}
\section{Noncommutative strong form of the Tate conjecture}\label{sec:alternative}
Let $k=\bbF_q$ be a finite field of characteristic $p$. In this section we prove that the noncommutative strong form of the Tate conjecture is equivalent to the noncommutative $p$-version of the Tate conjecture plus the noncommutative standard conjecture of type $D$. As a byproduct of this equivalence of conjectures, we obtain a proof of Theorem \ref{thm:strong} and also an alternative formulation of the noncommutative strong form of the Tate conjecture in terms of the enriched topological periodic cyclic homology functor (see \S\ref{sub:embedding}).
\subsection{Noncommutative standard conjecture of type $D$}\label{sub:standard}
Let $\cA$ be a smooth proper dg category. Similarly to \S\ref{sub:homological}, by combining the functor \eqref{eq:TP2} with the identification \eqref{eq:Homs} (with $\cA:=k$ and $\cB:=\cA$), we obtain an induced $\bbQ$-linear homomorphism $\mathrm{ch}\colon K_0(\cA)_\bbQ \to TP_0(\cA)_{1/p}$ and the associated homological Grothendieck group $K_0(\cA)_\bbQ/_{\!\sim\mathrm{hom}}$. Under these notations, the noncommutative standard conjecture of type $D$ asserts the following:

\vspace{0.1cm}

{\em Conjecture $\mathrm{D}_{\mathrm{nc}}(\cA)$: The equality $K_0(\cA)_\bbQ/_{\!\sim\mathrm{hom}}= K_0(\cA)_\bbQ/_{\!\sim \mathrm{num}}$ holds.}

\begin{remark}[Standard conjecture of type $D$]\label{rk:equivalence1}
Let $X$ be a smooth proper $k$-scheme. As proved in \cite[Thm.~1.1]{finite-CD}, we have the following equivalence of conjectures $\mathrm{D}_{\mathrm{nc}}(\perf_\dg(X))\Leftrightarrow \mathrm{D}(X)$, where $\mathrm{D}(X)$ stands for the standard conjecture of type $D$ (consult Grothendieck \cite{Grothendieck} and Kleiman \cite{Kleiman,Kleiman1}).
\end{remark}

\subsection{Noncommutative $p$-version of the Tate conjecture}\label{sub:p-Tate}
Let $\cA$ be a smooth proper dg category. As proved in \cite[Lem.~3.7]{HPD-Tate}, the $\bbQ$-linear homomorphism $\mathrm{ch}\colon K_0(\cA)_\bbQ \to TP_0(\cA)_{1/p}$ defined in \S\ref{sub:standard} takes values in $TP_0(\cA)_{1/p}^{\mathrm{F}_0}$. Hence, we can consider the induced $K$-linear homomorphism $\mathrm{ch}_K\colon K_0(\cA)_K \to TP_0(\cA)_{1/p}^{\mathrm{F}_0}$ and the associated homological Grothendieck group $K_0(\cA)_K/_{\!\sim\mathrm{hom}}$. Under these notations, the noncommutative $p$-version of the Tate conjecture asserts the following:

\vspace{0.1cm}

{\em Conjecture $\mathrm{T}^p_{\mathrm{nc}}(\cA)$: The homomorphism $\mathrm{ch}_K$ is surjective.}

\begin{remark}[$p$-version of the Tate conjecture]\label{rk:equivalence2}
Let $X$ be a smooth proper $k$-scheme. As proved in \cite[Thm.~1.3]{HPD-Tate}, we have the following equivalence of conjectures $\mathrm{T}^p_{\mathrm{nc}}(\perf_\dg(X))\Leftrightarrow \mathrm{T}^p(X)$, where $\mathrm{T}^p(X)$ stands for the $p$-version of the Tate conjecture (consult Milne \cite{Milne} and Tate \cite{Tate}).
\end{remark}
\subsection{Equivalence of conjectures}
The next result is of independent interest:
\begin{theorem}\label{thm:alternative}
Given a smooth proper $k$-linear dg category $\cA$, we have the equivalence of conjectures $\mathrm{ST}_{\mathrm{nc}}(\cA) \Leftrightarrow \mathrm{T}^p_{\mathrm{nc}}(\cA) +  \mathrm{D}_{\mathrm{nc}}(\cA)$.
\end{theorem}
\begin{proof}
We start by proving the implication $\mathrm{ST}_{\mathrm{nc}}(\cA) \Rightarrow  \mathrm{T}^p_{\mathrm{nc}}(\cA)+ \mathrm{D}_{\mathrm{nc}}(\cA)$. Recall from Remark \ref{rk:multiplicity} that if the conjecture $\mathrm{ST}_{\mathrm{nc}}(\cA)$ holds, then the algebraic multiplicity of the eigenvalue $1$ of $\mathrm{F}_0$ agrees with the dimension of the $\bbQ$-vector space $K_0(\cA)_\bbQ/_{\!\sim \mathrm{num}}$. Note also that the geometric multiplicity of the eigenvalue $1$ of $\mathrm{F}_0$, which is always less (or equal) than the algebraic multiplicity, agrees with the dimension of the $K$-vector space $TP_0(\cA)_{1/p}^{\mathrm{F}_0}$. In order to prove the conjecture $\mathrm{T}^p_{\mathrm{nc}}(\cA)$, we need then to show that the dimension of the $K$-vector space $TP_0(\cA)_{1/p}^{\mathrm{F}_0}$ is less (or equal) than the dimension of the $K$-vector space $K_0(\cA)_K/_{\!\sim \mathrm{hom}}$. This follows from the following (in)equalities: 
\begin{eqnarray*}
\mathrm{dim}_K TP_0(\cA)_{1/p}^{\mathrm{F}_0} & = & \mathrm{geometric}\,\mathrm{multiplicity}\,\mathrm{of}\,\mathrm{the}\,\mathrm{eigenvalue}\,1\,\mathrm{of}\,\mathrm{F}_0 \\
& \leq & \mathrm{algebraic}\,\mathrm{multiplicity}\,\mathrm{of}\,\mathrm{the}\,\mathrm{eigenvalue}\,1\,\mathrm{of}\,\mathrm{F}_0 \\
& = & \mathrm{dim}_\bbQ K_0(\cA)_\bbQ/_{\!\sim\mathrm{num}} \\
& \leq & \mathrm{dim}_\bbQ K_0(\cA)_\bbQ/_{\!\sim\mathrm{hom}} = \mathrm{dim}_K K_0(\cA)_K/_{\!\sim \mathrm{hom}}\,.
\end{eqnarray*}
Similarly, since $\mathrm{dim}_\bbQ K_0(\cA)_\bbQ/_{\!\sim\mathrm{num}} \leq \mathrm{dim}_\bbQ K_0(\cA)_\bbQ/_{\!\sim\mathrm{hom}}$, in order to prove the conjecture $\mathrm{D}_{\mathrm{nc}}(\cA)$, we need then to show that the dimension of the $\bbQ$-vector space $K_0(\cA)_\bbQ/_{\!\sim\mathrm{hom}}$ is less (or equal) than the dimension of the $\bbQ$-vector space $K_0(\cA)_\bbQ/_{\!\sim \mathrm{num}}$. This follows from the following (in)equalities:
\begin{eqnarray*}
\mathrm{dim}_\bbQ K_0(\cA)_\bbQ/_{\!\sim \mathrm{hom}} & = & \mathrm{dim}_K K_0(\cA)_K/_{\!\sim\mathrm{hom}} \\
& \leq & \mathrm{dim}_K TP_0(\cA)_{1/p}^{\mathrm{F}_0} \\
& = & \mathrm{geometric}\,\mathrm{multiplicity}\,\mathrm{of}\,\mathrm{the}\,\mathrm{eigenvalue}\,1\,\mathrm{of}\,\mathrm{F}_0 \\
& \leq &  \mathrm{algebraic}\,\mathrm{multiplicity}\,\mathrm{of}\,\mathrm{the}\,\mathrm{eigenvalue}\,1\,\mathrm{of}\,\mathrm{F}_0 \\
& = & \mathrm{dim}_\bbQ K_0(\cA)_\bbQ/_{\!\sim \mathrm{num}}\,.
\end{eqnarray*}

We now prove the implication $\mathrm{T}^p_{\mathrm{nc}}(\cA) + \mathrm{D}_{\mathrm{nc}}(\cA)  \Rightarrow \mathrm{ST}_{\mathrm{nc}}(\cA)$. Note that if both conjectures $\mathrm{T}^p_{\mathrm{nc}}(\cA)$ and $\mathrm{D}_{\mathrm{nc}}(\cA)$ hold, then the geometric multiplicity of the eigenvalue $1$ of $\mathrm{F}_0$ is equal to the dimension of the $\bbQ$-vector space $K_0(\cA)_\bbQ/_{\!\sim \mathrm{num}}$. Hence, in order to prove the conjecture $\mathrm{ST}_{\mathrm{nc}}(\cA)$, it suffices then to show that the geometric multiplicity of the eigenvalue $1$ of $\mathrm{F}_0$ agrees with the algebraic multiplicity of the eigenvalue $1$ of $\mathrm{F}_0$. Thanks to the general Lemma \ref{lem:general} below, this will follow from the injectivity of the canonical morphism $\epsilon\colon TP_0(\cA)_{1/p}^{\mathrm{F}_0} \to (TP_0(\cA)_{1/p})_{\mathrm{F}_0}$ (induced by the identity on $TP_0(\cA)_{1/p}$). Recall from \S\ref{sub:dg}-\S\ref{sub:NC} that $U(\cA)_\bbQ$ is a dualizable object of the symmetric monoidal category $\NChow(k)_\bbQ$ and that $U(\cA^\op)_\bbQ$ is the dual of $U(\cA)_\bbQ$. Consequently, by applying the functor $\Hom_{\NChow(k)_\bbQ}(U(k)_\bbQ,-)$ to the evaluation morphism $U(\cA^\op)_\bbQ \otimes U(\cA)_\bbQ \to U(k)_\bbQ$, we obtain (from the symmetric monoidal structure of $\NChow(k)_\bbQ$) a bilinear pairing $\psi\colon K_0(\cA^\op)_\bbQ \otimes_\bbQ K_0(\cA)_\bbQ \to \bbQ$. Note that since the $\bbQ$-linear functor \eqref{eq:TP2} is symmetric monoidal, we have the following commutative diagram
$$
\xymatrix{
TP_0(\cA^\op)_{1/p}^{\mathrm{F}_0} \otimes_K TP_0(\cA)^{\mathrm{F}_0}_{1/p} \ar[rr]^-{\theta_0} && K \\
K_0(\cA^\op)_K \otimes_K K_0(\cA)_K \ar[u]^-{\mathrm{ch}_K \otimes_K \mathrm{ch}_K} \ar[rr]_-{\psi_K}&& K \ar@{=}[u]\,,
}
$$
where $\theta_0$ stands for the perfect bilinear pairing of Proposition \ref{prop:pairings} and $\psi_K$ for the $K$-linearization of $\psi$. By adjunction, this yields the induced commutative diagram:
\begin{equation}\label{eq:square-key}
\xymatrix{
TP_0(\cA)_{1/p}^{\mathrm{F}_0} \ar[rr]^-{\theta_0^\natural} && \Hom_K(TP_0(\cA^\op)_{1/p}^{\mathrm{F}_0}, K)\ar[d]^-{\Hom_K(\mathrm{ch}_K,K)}\\
K_0(\cA)_K/_{\!\sim \mathrm{hom}} \ar[u]^-{\mathrm{ch}_K} \ar[rr]_-{\psi_K^\natural} && \Hom_K(K_0(\cA^\op)_K/_{\!\sim \mathrm{hom}}, K)\,.
}
\end{equation}
Thanks to the left-hand side commutative diagram of Proposition \ref{prop:pairings}, the morphism $\theta_0^\natural$ admits the following factorization:
$$ \theta_0^\natural\colon TP_0(\cA)_{1/p}^{\mathrm{F}_0} \stackrel{\epsilon}{\too} (TP_0(\cA)_{1/p})_{\mathrm{F}_0} \too \Hom_K(TP_0(\cA^\op)_{1/p}^{\mathrm{F}_0}, K)\,.$$
Using the fact that the left-hand side vertical morphism in \eqref{eq:square-key} is surjective (=conjecture $\mathrm{T}^p_{\mathrm{nc}}(\cA)$), we observe that in order to show that the canonical morphism $\epsilon$ is injective, it suffices then to show that the morphism $\psi_K^\natural$ is injective. As explained in \cite[\S6]{Compositio}, a Grothendieck class $\alpha \in K_0(\cA)_\bbQ$ is numerically trivial in the sense of \S\ref{sub:numerical} if and only if $\psi(\beta, \alpha)=0$ for every $\beta \in K_0(\cA^\op)_\bbQ$. In other words, the numerical Grothendieck group $K_0(\cA)_\bbQ/_{\!\sim\mathrm{num}}$ may be identified with the quotient of $K_0(\cA)_\bbQ/_{\!\sim\mathrm{hom}}$ by the kernel of $\theta_0^\natural$. Therefore, in order to prove that $\psi^\natural_K$ is injective, we can then consider the following commutative diagram:
$$
\xymatrix{
(K_0(\cA)_\bbQ/_{\!\sim\mathrm{hom}})_K \ar[d] \ar@{=}[r] & K_0(\cA)_K/_{\!\sim\mathrm{hom}}  \ar[d] \ar[r] & \Hom_K(K_0(\cA^\op)_K/_{\!\sim\mathrm{hom}},K) \\
(K_0(\cA)_\bbQ/_{\!\sim\mathrm{num}})_K  \ar@{=}[r]  & K_0(\cA)_K/_{\!\sim\mathrm{num}} \ar@/_1.2pc/[ur] &\,.
}
$$
Since the curved morphism is injective and the vertical morphism(s) is injective (=conjecture $\mathrm{D}_{\mathrm{nc}}(\cA)$), we hence conclude that the morphism $\psi^\natural_K$ is also injective. This finishes the proof of Theorem \ref{thm:alternative}.
\end{proof}
\begin{lemma}\label{lem:general}
Let $f\colon V \stackrel{\simeq}{\to} V$ be an automorphism of a finite-dimensional $K$-vector space $V$. Under these assumptions, the geometric multiplicity of the eigenvalue $1$ of $f$ agrees with the algebraic multiplicity of the eigenvalue $1$ of $f$ if and only if the canonical morphism $\epsilon\colon V^{f} \to V_{f}$ (induced by the identity on $V$) is injective.
\end{lemma}
\begin{proof}
Note that the geometric multiplicity of the eigenvalue $1$ of $f$ agrees with the algebraic multiplicity of the eigenvalue $1$ of $f$ if and only if $\mathrm{Ker}(\id - f)= \mathrm{Ker}((\id - f)^{\circ 2})$. Hence, the proof follows from the fact that the latter condition is equivalent to the condition $\mathrm{Ker}(\id - f)\cap \mathrm{Im}(\id - f)=\emptyset$, i.e., to the injectivity of the canonical morphism $V^{f} \to V_{f}$. 
\end{proof}

\subsection{Proof of Theorem \ref{thm:strong}}
As mentioned in Remarks \ref{rk:equivalence1} and \ref{rk:equivalence2}, we have the following equivalences of conjectures 
\begin{eqnarray*}
 \mathrm{D}_{\mathrm{nc}}(\perf_\dg(X))\Leftrightarrow \mathrm{D}(X)  && \mathrm{T}^p_{\mathrm{nc}}(\perf_\dg(X))\Leftrightarrow \mathrm{T}^p(X)\,.
\end{eqnarray*} 
Therefore, by combining Theorem \ref{thm:alternative} (with $\cA=\perf_\dg(X)$) with the equivalence of conjectures $\mathrm{ST}(X) \Leftrightarrow \mathrm{T}^p(X) + \mathrm{D}(X)$ established in \cite[Thm.~1.11]{Milne} (consult also \cite[Thm.~2.9]{Tate-motives}), we obtain the sought equivalence $\mathrm{ST}_{\mathrm{nc}}(\perf_\dg(X)) \Leftrightarrow \mathrm{ST}(X)$. 

\begin{remark}[Direct proof]
A direct proof of the following implication of conjectures $\mathrm{ST}(X) \Rightarrow \mathrm{ST}_{\mathrm{nc}}(\perf_\dg(X))$ can be achieved as follows: thanks to Corollary \ref{cor:order} and to the factorization \eqref{eq:factorization1}, we have the following equality:
$$ \mathrm{ord}_{s=0}\zeta_{\mathrm{even}}(\perf_\dg(X);s) = \sum_{0\leq j \leq d} \mathrm{ord}_{s=j}\zeta(X;s)\,.$$
Therefore, since the numerical Grothendieck group $K_0(\perf_\dg(X))_\bbQ/_{\!\sim \mathrm{num}}$ is isomorphic to the direct sum $\bigoplus_{i=0}^d \cZ^i(X)_\bbQ/_{\!\sim \mathrm{num}}$ (consult \cite[Prop.~1.7(i)]{Rigidity}), we conclude that the conjecture $\mathrm{ST}_{\mathrm{nc}}(\perf_\dg(X))$ follows from $\mathrm{ST}(X)$. 
\end{remark}
\subsection{Alternative formulation}\label{sub:embedding}
The next result, of independent result, provides an alternative formulation of the noncommutative strong form of the Tate conjecture (when all smooth proper dg categories are considered simultaneously):
\begin{proposition}
The enriched topological periodic cyclic homology functor \eqref{eq:enriched} induces the following $K$-linear symmetric monoidal fully-faithful functor
\begin{eqnarray}\label{eq:enriched2}
\NNum(k)_K \too \mathrm{Aut}(K[v^{\pm1}])^\tau && U(\cA)_K \mapsto (TP_\ast(\cA)_{1/p}, \mathrm{F}_\ast)
\end{eqnarray}
if and only if the conjecture $\mathrm{ST}_{\mathrm{nc}}(\cA)$ holds for every smooth proper dg category $\cA$.
\end{proposition}
\begin{proof} By construction, every object of $\NNum(k)_K$ is a direct summand of an object of the form $U(\cA)_K$ with $\cA$ a smooth proper $k$-linear dg category; see \cite[\S4.6]{book}. This implies that the functor \eqref{eq:enriched} descends to the category of noncommutative numerical motives $\NNum(k)_K$ if and only if the conjecture $\mathrm{D}_{\mathrm{nc}}(\cA)$ holds for every smooth proper dg category $\cA$. Given smooth proper dg categories $\cB$ and $\cC$, note that by definition of the category $\mathrm{Aut}(K[v^{\pm1}])^\tau$, we have an isomorphism: 
$$ \Hom_{\mathrm{Aut}(K[v^{\pm1}])^\tau}((TP_\ast(\cB)_{1/p}, \mathrm{F}_\ast),(TP_\ast(\cC)_{1/p}, \mathrm{F}_\ast)) \simeq TP_0(\cB^\op \otimes_k \cC)_{1/p}^{\mathrm{F}_0}\,.$$
This implies that the $K$-linear symmetric monoidal faithful functor \eqref{eq:enriched2} is moreover full if and only if the conjecture $\mathrm{T}_{\mathrm{nc}}^p(\cA)$ holds for every smooth proper dg category $\cA$. Consequently, the proof follows now from Theorem \ref{thm:alternative}.
\end{proof}

\section{Proof of Theorems \ref{thm:twist-new}, \ref{thm:hypersurface}-\ref{thm:orbifolds}, and \ref{thm:twist}-\ref{thm:finite}}
Let $N\!\!M \in \NChow(k)_\bbQ$ be a noncommutative Chow motive. Note that similarly to smooth proper dg categories, we can formulate the conjectures $\mathrm{W}_{\mathrm{nc}}(N\!\!M)$, $\mathrm{W}^l_{\mathrm{nc}}(N\!\!M)$, $\mathrm{ST}_{\mathrm{nc}}(N\!\!M)$, $\mathrm{D}_{\mathrm{nc}}(N\!\!M)$, and $\mathrm{T}^p_{\mathrm{nc}}(N\!\!M)$. Moreover, a proof similar to the one of Theorem \ref{thm:alternative} yields the equivalence $\mathrm{ST}_{\mathrm{nc}}(N\!\!M) \Leftrightarrow \mathrm{T}^p_{\mathrm{nc}}(N\!\!M) + \mathrm{D}_{\mathrm{nc}}(N\!\!M)$. 

In the particular case where $N\!\!M=U(\cA)_\bbQ$, with $\cA$ a smooth proper $k$-linear dg category, the aforementioned conjectures reduce to $\mathrm{W}_{\mathrm{nc}}(\cA)$, $\mathrm{W}^l_{\mathrm{nc}}(\cA)$, $\mathrm{ST}_{\mathrm{nc}}(\cA)$, $\mathrm{D}_{\mathrm{nc}}(\cA)$, and $\mathrm{T}^p_{\mathrm{nc}}(\cA)$, respectively.

\begin{proposition}\label{prop:properties}
The conjectures $\mathrm{W}_{\mathrm{nc}}(-)$, $\mathrm{W}^l_{\mathrm{nc}}(-)$, $\mathrm{ST}_{\mathrm{nc}}(-)$, $\mathrm{D}_{\mathrm{nc}}(-)$, $\mathrm{T}^p_{\mathrm{nc}}(-)$, are stable under direct sums and direct summands of noncommutative Chow motives.
\end{proposition}
\begin{proof}
The stability under direct sums is clear. The stability under direct summands is also clear for the noncommutative Weil conjecture(s), the noncommutative standard conjecture of type $D$, and the noncommutative $p$-version of the Tate conjecture. In what regards the noncommutative strong form of the Tate conjecture, it follows from the equivalence $\mathrm{ST}_{\mathrm{nc}}(N\!\!M) \Leftrightarrow \mathrm{T}^p_{\mathrm{nc}}(N\!\!M) + \mathrm{D}_{\mathrm{nc}}(N\!\!M)$.
\end{proof}
\begin{corollary}\label{cor:sums}
Given noncommutative Chow motives $N\!\!M, N\!\!M'\in \NChow(k)_\bbQ$, we have the equivalence of conjectures $\mathrm{W}_{\mathrm{nc}}(N\!\!M\oplus N\!\!M') \Leftrightarrow \mathrm{W}_{\mathrm{nc}}(N\!\!M) + \mathrm{W}_{\mathrm{nc}}(N\!\!M')$, the equivalence of conjectures $\mathrm{W}^l_{\mathrm{nc}}(N\!\!M\oplus N\!\!M') \Leftrightarrow \mathrm{W}^l_{\mathrm{nc}}(N\!\!M) + \mathrm{W}^l_{\mathrm{nc}}(N\!\!M')$, and the equivalence of conjectures $\mathrm{ST}_{\mathrm{nc}}(N\!\!M\oplus N\!\!M') \Leftrightarrow \mathrm{ST}_{\mathrm{nc}}(N\!\!M) + \mathrm{ST}_{\mathrm{nc}}(N\!\!M')$.
\end{corollary}
\begin{example}[Semi-orthogonal decompositions]\label{ex:semi}
Let $\cB, \cC \subseteq \cA$ be smooth proper $k$-linear dg categories inducing a semi-orthogonal decomposition of triangulated categories $\dgHo(\cA)=\langle \dgHo(\cB), \dgHo(\cC)\rangle$ in the sense of Bondal-Orlov \cite{BO}. As proved in \cite[Prop.~2.2]{book}, the inclusions $\cB, \cC \subseteq \cA$ give rise to an isomorphism of noncommutative Chow motives $U(\cA)_\bbQ \simeq U(\cB)_\bbQ \oplus U(\cC)_\bbQ$. Hence, Corollary \ref{cor:sums} yields the equivalence $\mathrm{W}_{\mathrm{nc}}(\cA) \Leftrightarrow \mathrm{W}_{\mathrm{nc}}(\cB) + \mathrm{W}_{\mathrm{nc}}(\cC)$, the equivalence $\mathrm{W}^l_{\mathrm{nc}}(\cA) \Leftrightarrow \mathrm{W}^l_{\mathrm{nc}}(\cB) + \mathrm{W}^l_{\mathrm{nc}}(\cC)$, and also the equivalence $\mathrm{ST}_{\mathrm{nc}}(\cA) \Leftrightarrow \mathrm{ST}_{\mathrm{nc}}(\cB) + \mathrm{ST}_{\mathrm{nc}}(\cC)$.
\end{example}

\subsection*{Proof of Theorem \ref{thm:twist-new}}
As proved in \cite[Thm.~2.1]{Azumaya}, we have an isomorphism of noncommutative Chow motives $U(\perf_\dg(X))_\bbQ \simeq U(\perf_\dg(\cX;\cF))_\bbQ$. Hence, we have the equivalence of conjectures $\mathrm{W}_{\mathrm{nc}}(\perf_\dg(X)) \Leftrightarrow\mathrm{W}_{\mathrm{nc}}(\perf_\dg(X;\cF))$, the equivalence of conjectures $\mathrm{W}^l_{\mathrm{nc}}(\perf_\dg(X)) \Leftrightarrow\mathrm{W}^l_{\mathrm{nc}}(\perf_\dg(X;\cF))$, and the equivalence of conjectures $\mathrm{ST}_{\mathrm{nc}}(\perf_\dg(X))\Leftrightarrow \mathrm{ST}_{\mathrm{nc}}(\perf_\dg(X;\cF))$. Therefore, the proof follows from Theorems \ref{thm:main}, \ref{thm:new} (plus Remark \ref{rk:new}), and \ref{thm:strong}.

\subsection*{Proof of Theorem \ref{thm:hypersurface}}
The (noncommutative) Weil conjecture(s) as well as the (noncommutative) strong form of the Tate conjecture hold for $\mathrm{Spec}(k)$. Consequently, an iterated application of Example \ref{ex:semi} to the semi-orthogonal decomposition \eqref{eq:semi-hyper} yields the equivalence of conjectures $\mathrm{W}_{\mathrm{nc}}(\perf_\dg(X)) \Leftrightarrow \mathrm{W}_{\mathrm{nc}}(\cT_\dg(X))$, the equivalence of conjectures $\mathrm{W}^l_{\mathrm{nc}}(\perf_\dg(X)) \Leftrightarrow \mathrm{W}^l_{\mathrm{nc}}(\cT_\dg(X))$, and the equivalence of conjectures $\mathrm{ST}_{\mathrm{nc}}(\perf_\dg(X)) \Leftrightarrow \mathrm{ST}_{\mathrm{nc}}(\cT_\dg(X))$. Therefore, the proof follows from Theorems \ref{thm:main}, \ref{thm:new} (plus Remark \ref{rk:new}), and \ref{thm:strong}.

\subsection*{Proof of Theorem \ref{thm:glue}}
Up to Morita equivalence, the dg category $X \ominus_{\mathrm{B}}Y$ admits a semi-orthogonal decomposition whose components are (Fourier-Mukai) equivalent to $\perf(X)$ and $\perf(Y)$. Consequently, the proof of Theorem \ref{thm:glue} follows from the combination of Example \ref{ex:semi} with Theorems \ref{thm:main}, \ref{thm:new} (plus Remark \ref{rk:new}),~and~\ref{thm:strong}.

\subsection*{Proof of Theorem \ref{thm:root}}
As proved by Ishii-Ueda in \cite[Thm.~1.6]{IU}, whenever the zero locus $D \hookrightarrow X$ of $\varsigma$ is smooth, we have a semi-orthogonal decomposition
\begin{equation}\label{eq:semi-root}
\perf(\cX)= \langle \perf(D)_{n-1}, \ldots, \perf(D)_1, f^\ast(\perf(X))\rangle\,,
\end{equation}
where all the categories $\perf(D)_i$ are (Fourier-Mukai) equivalent to $\perf(D)$ and $f^\ast(\perf(X))$ is (Fourier-Mukai) equivalent to $\perf(X)$. Therefore, making use of Theorems \ref{thm:main}, \ref{thm:new} (plus Remark \ref{rk:new}), and \ref{thm:strong}, the proof follows from an iterated application of Example \ref{ex:semi} to the semi-orthogonal decomposition \eqref{eq:semi-root}.

\subsection*{Proof of Theorem \ref{thm:orbifolds}}
Let $\varphi$ be the set of all cyclic subgroups of $G$ and $\varphi_{\!/\!\sim}$ a set of representatives of conjugacy classes in $\varphi$. Since the category $\NChow(k)_\bbQ$ is $\bbQ$-linear, it follows from \cite[Thm.~1.1 and Rk.~1.3(iii)]{Orbifold} that the noncommutative Chow motive $U(\perf_\dg(\cX))_\bbQ$ is a direct summand of the following direct sum:
\begin{equation}\label{eq:direct}
\bigoplus_{\sigma \in \varphi_{\!/\!\sim}} U(\perf_\dg(X^\sigma \times \mathrm{Spec}(k[\sigma])))_\bbQ\,.
\end{equation}
Under the assumption $n|(q-1)$, the same holds with $X^\sigma \times \mathrm{Spec}(k[\sigma])$ replaced by $X^\sigma$; consult \cite[Cor.~1.5(ii)]{Orbifold}. Recall from Proposition \ref{prop:properties} that the noncommutative Weil conjecture(s) and the noncommutative strong form of the Tate conjecture are stable under direct sums and direct summands of noncommutative Chow motives. Therefore, the proof follows from the combination of \eqref{eq:direct} (under the assumption $n|(q-1)$, replace $X^\sigma \times \mathrm{Spec}(k[\sigma])$ by $X^\sigma$) with Theorems \ref{thm:main}, \ref{thm:new} (plus Remark \ref{rk:new}), and \ref{thm:strong}.

\subsection*{Proof of Theorem \ref{thm:twist}}
The proof of Theorem \ref{thm:twist} is similar to the proof of Theorem \ref{thm:orbifolds} (in the case where $n|(q-1)$). Simply, replace $\perf_\dg(\cX)$ by $\perf_\dg(\cX;\cF)$, $\perf_\dg(X^\sigma)$ by $\perf_\dg(Y_\sigma)$, and \cite[Cor.~1.5(ii)]{Orbifold} by \cite[Cor.~1.28(ii)]{Orbifold}. 

\subsection*{Proof of Theorem \ref{thm:dg}}
Consider the finite-dimensional $k$-algebra $\dgHo(A)$ and the associated semi-simple $k$-algebra $B:=\dgHo(A)/\mathrm{Jac}(\dgHo(A))$, where $\mathrm{Jac}(\dgHo(A))$ stands for the Jacobson radical of $\dgHo(A)$. Let us write $V_1, \ldots, V_n$ for the simple (right) $B$-modules and $D_1:=\mathrm{End}_{B}(V_1), \ldots, D_n:=\mathrm{End}_{B}(V_n)$ for the associated division $k$-algebras. Note that since $k$ is finite, the division $k$-algebras $D_1, \ldots, D_n$ are just finite field extensions $\kappa_1, \ldots, \kappa_n$ of $k$. Under these notations, we have isomorphisms of noncommutative Chow motives
\begin{equation}\label{eq:iso-finite}
U(A)_\bbQ \simeq U(B)_\bbQ \simeq U(\kappa_1)_\bbQ \oplus \cdots \oplus U(\kappa_n)_\bbQ\,,
\end{equation}
where the first isomorphism was established by Raedschelders-Stevenson in \cite[Thm.~3.5]{RS} and the second isomorphism follows from the fact that $B$ is semi-simple. By combining the isomorphism \eqref{eq:iso-finite} with Corollary \ref{cor:sums}, we obtain the equivalence $\mathrm{W}_{\mathrm{nc}}(A) \Leftrightarrow \mathrm{W}_{\mathrm{nc}}(\kappa_1) + \cdots +  \mathrm{W}_{\mathrm{nc}}(\kappa_n)$, the equivalence $\mathrm{W}^l_{\mathrm{nc}}(A) \Leftrightarrow \mathrm{W}^l_{\mathrm{nc}}(\kappa_1) + \cdots +  \mathrm{W}^l_{\mathrm{nc}}(\kappa_n)$, and the equivalence $\mathrm{ST}_{\mathrm{nc}}(A) \Leftrightarrow \mathrm{ST}_{\mathrm{nc}}(\kappa_1) + \cdots + \mathrm{ST}_{\mathrm{nc}}(\kappa_n)$. Consequently, since the conjectures $\mathrm{W}_{\mathrm{nc}}(\kappa_i)$, $\mathrm{W}^l_{\mathrm{nc}}(\kappa_i)$, and $\mathrm{ST}_{\mathrm{nc}}(\kappa_i)$, hold, the conjectures $\mathrm{W}_{\mathrm{nc}}(A)$, $\mathrm{W}^l_{\mathrm{nc}}(A)$, and $\mathrm{ST}_{\mathrm{nc}}(A)$, also hold.

\subsection*{Proof of Theorem \ref{thm:finite}}
Following Orlov \cite[\S2.3]{Orlov1}, let us write $\mathrm{Aus}(A)$ for the smooth proper Auslander dg $k$-algebra associated to $A$ (Orlov used a different notation). As proved in \cite[Thms.~2.18-2.19]{Orlov1}, we have a semi-orthogonal decomposition
$$ \perf(\mathrm{Aus}(A))= \langle \perf(D_1), \ldots, \perf(D_n)\rangle\,,$$
where $D_i$ is a division $k$-algebra; note that since $k$ is finite, the division $k$-algebras $D_1, \ldots, D_n$ are just finite field extensions $\kappa_1, \ldots, \kappa_n$ of $k$. Moreover, as also proved in {\em loc. cit.}, the category $\mathrm{perf}(A)$ can be embedded (using a Fourier-Mukai functor) as an admissible triangulated subcategory of $\perf(\mathrm{Aus}(A))$. This implies that the noncommutative Chow motive $U(A)_\bbQ$ is a direct summand of the direct sum $\bigoplus_{i=1}^n U(\kappa_i)$. Consequently, since the conjectures $\mathrm{W}_{\mathrm{nc}}(\kappa_i)$, $\mathrm{W}^l_{\mathrm{nc}}(\kappa_i)$, and $\mathrm{ST}_{\mathrm{nc}}(\kappa_i)$, hold, we conclude from Proposition \ref{prop:properties} that the conjectures $\mathrm{W}_{\mathrm{nc}}(A)$, $\mathrm{W}^l_{\mathrm{nc}}(A)$, and $\mathrm{ST}_{\mathrm{nc}}(A)$, also hold.

\section{Proof of Theorem \ref{thm:quadrics}}
As proved in \cite[Thm.~5.5]{Kuznetsov-Quadrics} (see also \cite[Thm.~2.3.7]{Auel}), we have the following semi-orthogonal decomposition 
\begin{equation*}\label{eq:Kuznetsov}
\perf(X)=\langle \perf(\bbP^1; \cC l_0(q)), \cO_X(1), \ldots, \cO_X(n-4)\rangle\,,
\end{equation*}
where $\cC l_0(q)$ stands for the sheaf of even parts of the Clifford algebra associated to the flat quadric fibration $f\colon Q \to \bbP^1$. Consequently, since the (noncommutative) Weil conjecture as well as the (noncommutative) strong form of the Tate conjecture hold for $\mathrm{Spec}(k)$, an iterated application of Example \ref{ex:semi} yields the following equivalences of conjectures:
\begin{eqnarray}
\mathrm{W}_{\mathrm{nc}}(\perf_\dg(X)) & \Leftrightarrow & \mathrm{W}_{\mathrm{nc}}(\perf_\dg(\bbP^1;\cC l_0(q))) \label{eq:Clifford1}\\
\mathrm{ST}_{\mathrm{nc}}(\perf_\dg(X)) & \Leftrightarrow & \mathrm{ST}_{\mathrm{nc}}(\perf_\dg(\bbP^1;\cC l_0(q))) \,.\label{eq:Clifford2}
\end{eqnarray}
We start by proving item (i). Following \cite[\S3.5]{Kuznetsov-Quadrics} (see also \cite[\S1.6]{Auel}), let $\cZ$ be the center of $\cC l_0(q)$ and $\mathrm{Spec}(\cZ)=:\widetilde{\bbP}^1 \to \bbP^1$ the discriminant cover of $\bbP^1$. As explained in {\em loc. cit.}, $\widetilde{\bbP}^1 \to \bbP^1$ is a $2$-fold cover which is ramified over the (finite) set $D$ of critical values of $f$. Moreover, since $D$ is smooth, $\widetilde{\bbP}^1$ is also smooth. Let us write $\cF$ for the sheaf $\cC l_0(q)$ considered as a sheaf of noncommutative algebras over $\widetilde{\bbP}^1$. As proved in {\em loc. cit.}, since by assumption all the fibers of $f\colon Q \to \bbP^1$ have corank $\leq 1$, $\cF$ is a sheaf of Azumaya algebras over $\widetilde{\bbP}^1$. Moreover, the category $\perf(\bbP^1; \cC l_0(q))$ is (Fourier-Mukai) equivalent to $\perf(\widetilde{\bbP}^1; \cF)$. Note that since the Brauer group of every smooth curve over a finite field $k$ is trivial, the latter category reduces to $\perf(\widetilde{\bbP}^1)$. Therefore, thanks to Theorem \ref{thm:main}, resp. Theorem \ref{thm:strong}, the equivalence \eqref{eq:Clifford1}, resp. \eqref{eq:Clifford2}, reduces to $\mathrm{W}(X)\Leftrightarrow \mathrm{W}(\widetilde{\bbP}^1)$, resp. $\mathrm{ST}(X)\Leftrightarrow \mathrm{ST}(\widetilde{\bbP}^1)$. Finally, since $\widetilde{\bbP}^1$ is a curve, we hence conclude that the conjectures $\mathrm{W}(X)$ and $\mathrm{ST}(X)$ hold. 

We now prove item (ii). Note first that $1/2 \in k$. Following \cite[\S3.6]{Kuznetsov-Quadrics} (see also \cite[\S1.7]{Auel}), let $\widehat{\bbP}^1$ the discriminant stack associated to the flat quadric fibration $f \colon Q \to \bbP^1$. As explained in {\em loc. cit.}, since $1/2\in k$, $\widehat{\bbP}^1$ is a square root stack. Moreover, the underlying $k$-scheme of $\widehat{\bbP}^1$ is $\bbP^1$. Therefore, it follows from Theorem \ref{thm:root} that the conjectures $\mathrm{W}_{\mathrm{nc}}(\perf_\dg(\widehat{\bbP}^1))$ and $\mathrm{ST}_{\mathrm{nc}}(\perf_\dg(\widehat{\bbP}^1))$ hold. Let us write $\cF$ for the sheaf $\cC l_0(q)$ considered as a sheaf of noncommutative algebras over $\widehat{\bbP}^1$. As proved in {\em loc. cit.}, since by assumption all the fibers of $f\colon Q \to \bbP^1$ have corank $\leq 1$, $\cF$ is a sheaf of Azumaya algebras over $\widehat{\bbP}^1$. Moreover, the category $\perf(\bbP^1; \cC l_0(q))$ is (Fourier-Mukai) equivalent to $\perf(\widehat{\bbP}^1; \cF)$. Note that since the Brauer group of every smooth curve over a finite field $k$ is trivial, the latter category reduces to $\perf(\widehat{\bbP}^1)$. Hence, thanks to Theorem \ref{thm:main}, resp. Theorem \ref{thm:strong}, the equivalence \eqref{eq:Clifford1}, resp. \eqref{eq:Clifford2}, reduces to $\mathrm{W}(X) \Leftrightarrow \mathrm{W}_{\mathrm{nc}}(\perf_\dg(\widehat{\bbP}^1))$, resp. $\mathrm{ST}(X) \Leftrightarrow \mathrm{ST}_{\mathrm{nc}}(\perf_\dg(\widehat{\bbP}^1))$. Finally, since $\mathrm{W}_{\mathrm{nc}}(\perf_\dg(\widehat{\bbP}^1))$ and $\mathrm{ST}_{\mathrm{nc}}(\perf_\dg(\widehat{\bbP}^1))$ hold, we hence conclude that the conjectures $\mathrm{W}(X)$ and $\mathrm{ST}(X)$ also hold.
\section{Proof of Theorem \ref{thm:HPD-duality}}

As proved in \cite[Cor.~3.7]{HPD}, the following holds:
\begin{itemize}
\item[(a)] When $\mathrm{dim}(L)< d_2r$, the category $\perf(X_L)$ admits a semi-orthogonal decomposition with one component (Fourier-Mukai) equivalent to $\perf(Y_L)$ and with $(d_2r - \mathrm{dim}(L))\binom{d_1}{r}$ exceptional objects.
\item[(b)] When $\mathrm{dim}(L)=d_2r$, the category $\perf(X_L)$ is (Fourier-Mukai) equivalent to the category $\perf(Y_L)$.
\item[(c)] When $\mathrm{dim}(L)> d_2r$, the category $\perf(Y_L)$ admits a semi-orthogonal decomposition with one component (Fourier-Mukai) equivalent to $\perf(X_L)$ and with $(\mathrm{dim}(L)-d_2r)\binom{d_1}{r}$ exceptional objects.
\end{itemize}
Consequently, since the (noncommutative) Weil conjecture as well as the (noncommutative) strong form of the Tate conjecture hold for $\mathrm{Spec}(k)$, by combining Theorems \ref{thm:main} and \ref{thm:strong} with an iterated application of Example \ref{ex:semi}, we hence conclude from (a)-(c) that $\mathrm{W}(X_L) \Leftrightarrow \mathrm{W}(Y_L)$ and $\mathrm{ST}(X_L) \Leftrightarrow \mathrm{ST}(Y_L)$. Therefore, the proof of item (i), resp. item (ii), follows from the fact that $\mathrm{dim}(X_L)=r(d_1+ d_2 -r) - 1 - \mathrm{dim}(L)$, resp. $\mathrm{dim}(Y_L)=r(d_1 - d_2 -r)-1 + \mathrm{dim}(L)$.

\section{Proof of Theorem \ref{thm:main-L-functions}}\label{proof:main-L-functions}
Given a prime number $p\neq p_1, \ldots, p_m$, let us write $\chi_{(0,p)}:=\mathrm{dim}_K TP_0(\mathfrak{A}_p)_{1/p}$ and $\chi_{(1,p)}:=\mathrm{dim}_K TP_1(\mathfrak{A}_p)_{1/p}$. In the same vein, given $i=1, \ldots, m$, let us write $\chi_{(0,p_i)}:=\mathrm{dim}_K TP_0(\mathfrak{A}_{p_i})_{1/p_i}$ and $\chi_{(1,p_i)}:=\mathrm{dim}_K TP_1(\mathfrak{A}_{p_i})_{1/{p_i}}$. The next result, of independent interest, provides a uniform upper bound for these dimensions.
\begin{proposition}\label{prop:bound-11}
There exists an integer $C_0\gg 0$, resp. $C_1\gg 0$, such that $\chi_{(0,p)}\leq C_0$ and $\chi_{(0,p_i)}\leq C_0$, resp. $\chi_{(1,p)}\leq C_1$ and $\chi_{(1,p_i)}\leq C_1$, for every prime $p\neq p_1, \ldots, p_m$ and $i=1, \ldots, m$.
\end{proposition}
\begin{proof}
Following \cite[\S2.28]{book}, Hochschild homology gives rise to a symmetric monoidal functor $HH\colon \dgcat_{\mathrm{sp}}(\bbZ[1/p_1, \ldots, 1/p_m]) \to \cD_c(\bbZ[1/p_1, \cdots, 1/p_m])$. Since $\mathfrak{A}$ is a smooth proper $\bbZ[1/p_1, \cdots, 1/p_m]$-linear dg category, this implies not only that the Hochschild homology modules $HH_n(\mathfrak{A}), n \in \bbZ$, are finitely generated but moreover that they are zero for $|n|\gg 0$. Let us choose a finite set of generators of $HH_n(\mathfrak{A})$ and write $\#_n$ for its cardinality. Under these choices, we set $C'_0:=\sum_{n\,\mathrm{even}}(\#_n + \#_{n-1})$ and $C'_1:=\sum_{n \,\mathrm{odd}}(\#_n + \#_{n-1})$. 

Choose a prime $p\neq p_1, \ldots, p_m$ and consider the associated symmetric monoidal functor between rigid symmetric monoidal categories:
\begin{equation}\label{eq:base-change}
-\otimes^{\bf L}_{\bbZ[1/p_1, \ldots, 1/p_m]} \bbF_p\colon \Hmo_{\mathrm{sp}}(\bbZ[1/p_1, \ldots, 1/p_m]) \too \Hmo_{\mathrm{sp}}(\bbF_p)\,.
\end{equation} 
As proved in \cite[Prop.~2.24]{book}, the Hochschild homology $HH(\mathfrak{A})$, resp. $HH(\mathfrak{A}_p)$, may be understood as the Euler characteristic of $\mathfrak{A}$, resp. $\mathfrak{A}_p$, in the rigid symmetric monoidal category $\Hmo_{\mathrm{sp}}(\bbZ[1/p_1, \ldots, 1/p_m])$, resp. $\Hmo_{\mathrm{sp}}(\bbF_p)$. Consequently, the above functor \eqref{eq:base-change} yields an isomorphism $HH(\mathfrak{A}_p)\simeq HH(\mathfrak{A})\otimes^{\bf L}_{\bbZ[1/p_1, \ldots, 1/p_m]} \bbF_p$ in the derived category $\cD_c(\bbF_p)$. Using the following free resolution 
$$ 0 \too \bbZ[1/p_1, \ldots, 1/p_m] \stackrel{\cdot p}{\too} \bbZ[1/p_1, \ldots, 1/p_m] \too \bbF_p \too 0\,,$$
we hence obtain the K\"unneth (split) short exact sequence of $\bbF_p$-vector spaces
$$ 0 \too HH_n(\mathfrak{A})\otimes_{\bbZ[1/p_1, \ldots, 1/p_m]} \bbF_p \too HH_n(\mathfrak{A}_p) \too \mathrm{Tor}_p(HH_{n-1}(\mathfrak{A}))\too 0\,,$$
where $\mathrm{Tor}_p(HH_{n-1}(\mathfrak{A}))$ stands for the $p$-torsion subgroup of $HH_{n-1}(\mathfrak{A})$. This (split) short exact sequence naturally implies the following (in)equality:
\begin{eqnarray*}
\mathrm{dim}_{\bbF_p}HH_n(\mathfrak{A}_p) & = & \mathrm{dim}_{\bbF_p}(HH_n(\mathfrak{A}) \otimes_{\bbZ[1/p_1, \ldots, 1/p_m]} \bbF_p) + \mathrm{dim}_{\bbF_p}\mathrm{Tor}_p(HH_{n-1}(\mathfrak{A})) \nonumber \\
& \leq & \#_n + \#_{n-1}\,. \label{eq:inequality}
\end{eqnarray*}

Similarly to \S\ref{sub:homological}, periodic cyclic homology (over a finite field) gives rise to a functor $HP_\ast(-)\colon \dgcat_{\mathrm{sp}}(\bbF_p) \to \mathrm{mod}_\bbZ(\bbF_p[v^{\pm1}])$ with values in the category of (degreewise finite-dimensional) $\bbZ$-graded $\bbF_p[v^{\pm1}]$-modules, where $v$ is a variable of degree $-2$. Thanks to the (convergent) Hodge-to-de Rham spectral sequence $HH_\ast(\mathfrak{A}_p)[u^{\pm 1}] \Rightarrow HP_\ast(\mathfrak{A}_p)$, where $u$ is a formal variable of degree $-2$, we have:
\begin{eqnarray*}
\mathrm{dim}_{\bbF_p}HP_0(\mathfrak{A}_p) \leq \sum_{n \, \mathrm{even}} \mathrm{dim}_{\bbF_p}HH_n(\mathfrak{A}_p) && \mathrm{dim}_{\bbF_p}HP_1(\mathfrak{A}_p) \leq \sum_{n \, \mathrm{odd}} \mathrm{dim}_{\bbF_p}HH_n(\mathfrak{A}_p)\,.
\end{eqnarray*}
By combining these inequalities with the fact that $\mathrm{dim}_{\bbF_p}HH_n(\mathfrak{A}_p) \leq \#_n + \#_{n-1}$, we hence conclude that $\mathrm{dim}_{\bbF_p}HP_0(\mathfrak{A}_p) \leq C'_0$ and $\mathrm{dim}_{\bbF_p}HP_1(\mathfrak{A}_p) \leq C'_1$.

Now, recall from \S\ref{sec:cyclotomic} that $TP_\ast(\mathfrak{A}_p)$ is a (degreewise finitely-generated) $\bbZ$-graded module over $TP_\ast(k)\simeq \bbZ_p[v^{\pm1}]$, where $v$ is a variable of degree $-2$. As proved in \cite[Thm.~3.4]{Ahkil} (see also \cite{BMS}), in the same way that $\bbZ_p/p\simeq \bbF_p$, we have natural isomorphisms $\pi_\ast(TP(\mathfrak{A}_p)/p)\simeq HP_\ast(\mathfrak{A}_p)$. Via the inclusion $TP_\ast(\mathfrak{A}_p)/p \subseteq \pi_\ast(TP(\mathfrak{A}_p)/p)$, we hence conclude that $TP_0(\mathfrak{A}_p)/p$ and $TP_1(\mathfrak{A}_p)/p$ are finite-dimensional $\bbF_p$-vector spaces. Let $\chi'_{(0,p)}$ and $\chi'_{(1,p)}$ be their dimensions. It follows from \cite[Lem.~2.12]{Ahkil} that the $\bbZ_p$-module $TP_0(\mathfrak{A}_p)$, resp. $TP_1(\mathfrak{A}_p)$, is a quotient of the free $\bbZ_p$-module of rank $\chi'_{(0,p)}$, resp. $\chi'_{(1,p)}$. Therefore, after inverting $p$, we obtain the inequalities:
\begin{eqnarray*}
\chi_{(0,p)}\leq \chi'_{(0,p)} \leq \mathrm{dim}_{\bbF_p}HP_0(\mathfrak{A}_p)\leq C'_0 && \chi_{(1,p)}\leq \chi'_{(1,p)} \leq \mathrm{dim}_{\bbF_p}HP_1(\mathfrak{A}_p)\leq C'_1\,.
\end{eqnarray*}
The proof follows now from the definitions $C_0:=\{C'_0, \chi_{(0, p_1)}, \ldots, \chi_{(0,p_m)}\}$ and $C_1:=\mathrm{max}\{C'_1, \chi_{(1,p_1)}, \ldots, \chi_{(1,p_m)}\}$.
\end{proof}
Given a prime $p\neq p_1, \ldots p_m$ and an integer $n\geq 1$, consider the complex numbers
\begin{eqnarray*}
\#_{(0,p,n)} & := & \mathrm{trace}(\mathrm{F}_0^{\circ n} \otimes_{K, \iota} \bbC\,|\, TP_0(\mathfrak{A}_p)_{1/p} \otimes_{K, \iota} \bbC) \\
\#_{(1,p,n)} & := & \mathrm{trace}(\mathrm{F}_1^{\circ n} \otimes_{K, \iota} \bbC\,|\, TP_1(\mathfrak{A}_p)_{1/p} \otimes_{K, \iota} \bbC) \,,
\end{eqnarray*}
where $\mathrm{F}_\ast^{\circ n}$ stands for the $n$-fold composition of the cyclotomic Frobenius $\mathrm{F}_\ast$ (consult Notation \ref{not:Frobenius}). In the same vein, given $i=1, \ldots, m$, consider the complex numbers:
\begin{eqnarray*}
\#_{(0,p_i,n)} & := & \mathrm{trace}(\mathrm{F}_0^{\circ n} \otimes_{K, \iota} \bbC\,|\, TP_0(\mathfrak{A}_{p_i})_{1/p_i} \otimes_{K, \iota} \bbC) \\
\#_{(1,p_i,n)} & := & \mathrm{trace}(\mathrm{F}_1^{\circ n} \otimes_{K, \iota} \bbC\,|\, TP_1(\mathfrak{A}_{p_i})_{1/p_i} \otimes_{K, \iota} \bbC) \,.
\end{eqnarray*}

\begin{proposition}\label{prop:bound2}
Assume that the conjectures $\mathrm{W}_{\mathrm{nc}}(\mathfrak{A}_p)$, resp. $\mathrm{W}_{\mathrm{nc}}(\mathfrak{A}_{p_i})$, holds. Under this assumption, there exist integers $C_0, C_1 \gg 0$ such that $|\#_{(0,p,n)}|\leq C_0$ and $|\#_{(1,p,n)}|\leq C_1 p^{\frac{1}{2}n}$, resp. $|\#_{(0,p_i,n)}|\leq C_0$ and $|\#_{(1,p_i,n)}|\leq C_1 p_i^{\frac{1}{2}n}$, for every prime $p\neq p_1, \ldots, p_m$ and $n\geq1$, resp. for every $i=1, \ldots, m$ and $n\geq1$.
\end{proposition}
\begin{proof}
Given a prime $p\neq p_1, \ldots, p_m$, let us write $\{\lambda_{(0,p,1)}, \ldots, \lambda_{(0,p,\chi_{(0,p)})}\}$ for the eigenvalues (with multiplicity) of the automorphism $\mathrm{F}_0\otimes_{K, \iota} \bbC$ of $TP_0(\mathfrak{A}_p)_{1/p} \otimes_{K, \iota}\bbC$ and $\{\lambda_{(1,p,1)}, \ldots, \lambda_{(1,p,\chi_{(1,p)})}\}$ for the eigenvalues (with multiplicity) of the automorphism $\mathrm{F}_1\otimes_{K, \iota}\bbC$ of $TP_1(\mathfrak{A}_p)_{1/p} \otimes_{K, \iota}\bbC$. In the same vein, given $i=1, \ldots m$, let us write $\{\lambda_{(0,p_i,1)}, \ldots, \lambda_{(0,p_i,\chi_{(0,p_i)})}\}$ for the eigenvalues (with multiplicity) of the automorphism $\mathrm{F}_0\otimes_{K, \iota} \bbC$ of $TP_0(\mathfrak{A}_{p_i})_{1/{p_i}} \otimes_{K, \iota}\bbC$ and $\{\lambda_{(1,p_i,1)}, \ldots, \lambda_{(1,p_i,\chi_{(1,p_i)})}\}$, for the eigenvalues of the automorphism $\mathrm{F}_1\otimes_{K, \iota}\bbC$ of $TP_1(\mathfrak{A}_{p_i})_{1/{p_i}} \otimes_{K, \iota} \bbC$. Under these notations, we have the following (in)equalities
\begin{eqnarray}
|\#_{(0,p,n)}|& = & |\mathrm{trace}(\mathrm{F}_0^{\circ n}\otimes_{K, \iota}\bbC)| \nonumber\\
& = &  |\lambda^n_{(0,p,1))}+ \cdots + \lambda^n_{(0,p,\chi_{(0,p)})}| \nonumber\\
& \leq & |\lambda_{(0,p,1))}|^n + \cdots + |\lambda_{(0,p,\chi_{(0,p)})}|^n \nonumber \\
& = & \chi_{(0,p)} \label{eq:inequality1} \\
& \leq  & C_0 \label{eq:inequality11}\,,
\end{eqnarray}
where \eqref{eq:inequality1} follows from conjecture $\mathrm{W}_{\mathrm{nc}}(\mathfrak{A}_p)$ and \eqref{eq:inequality11} from  Proposition \ref{prop:bound-11}; similarly with $p$ replaced by $p_i$. In the same vein, we have the (in)equalities
\begin{eqnarray}
|\#_{(1,p,n)}|& = & |\mathrm{trace}(\mathrm{F}_1^{\circ n}\otimes_{K, \iota}\bbC)| \nonumber\\
& = &  |\lambda^n_{(1,p,1))}+ \cdots + \lambda^n_{(1,p,\chi_{(1,p)})}| \nonumber\\
& \leq & |\lambda_{(1,p,1))}|^n + \cdots + |\lambda_{(1,p,\chi_{(1,p)})}|^n \nonumber \\
& = & \chi_{(1,p)} p^{\frac{1}{2}n}  \label{eq:inequality2} \\
& \leq & C_1 p^{\frac{1}{2}n} \label{eq:inequality22}\,,
\end{eqnarray}
where \eqref{eq:inequality2} follows from conjecture $\mathrm{W}_{\mathrm{nc}}(\mathfrak{A}_p)$ and \eqref{eq:inequality22} from Proposition \ref{prop:bound-11}; similarly with $p$ replaced by $p_i$.
\end{proof}
Recall the following general result, whose proof is a simple linear algebra exercise that we leave for the reader. 
\begin{lemma}\label{lem:exercise1}
Given an endomorphism $f\colon V \to V$ of a finite-dimensional vector space, we have the following equality of formal power series
$$ \mathrm{log}(\frac{1}{\mathrm{det}(\id - tf|V)})= \sum_{n\geq 1} \mathrm{trace}(f^{\circ n})\frac{t^n}{n}\,,$$
where $\mathrm{log}(t):= \sum_{n\geq1} \frac{(-1)^{n+1}}{n}(t-1)^n$.
\end{lemma}
Given a prime $p\neq p_1, \ldots, p_m$, consider the following (auxiliar) formal power series $\phi_{(0,p)}(t) :=\sum_{n\geq 1} \#_{(0,p,n)}\frac{t^n}{n}$ and $\phi_{(1,p)}(t) := \sum_{n\geq 1} \#_{(1,p,n)}\frac{t^n}{n}$ as well as their exponentiations $\varphi_{(0,p)}(t):=\sum_{n\geq 0} a_{(0,p,n)} t^n$ and $\varphi_{(1,p)}(t):=\sum_{n\geq 0} a_{(1,p,n)} t^n$. In the same vein, given $i=1,\ldots, m$, consider the following (auxiliar) formal power series $\phi_{(0,p_i)}(t) :=\sum_{n\geq 1} \#_{(0,p_i,n)}\frac{t^n}{n}$ and $\phi_{(1,p_i)}(t) := \sum_{n\geq 1} \#_{(1,p_i,n)}\frac{t^n}{n}$ as well as their exponentiations $\varphi_{(0,p_i)}(t):=\sum_{n\geq 0} a_{(0,p_i,n)} t^n$ and $\varphi_{(1,p_i)}(t):=\sum_{n\geq 0} a_{(1,p_i,n)} t^n$. Note that thanks to the above Lemma \ref{lem:exercise1}, we have the following (formal) equalities:
\begin{eqnarray}
\varphi_{(0,p)}(p^{-s})=\zeta_{\mathrm{even}}(\mathfrak{A}_p;s) && \varphi_{(0,p_i)}(p_i^{-s})=\zeta_{\mathrm{even}}(\mathfrak{A}_{p_i};s) \label{eq:Euler1}\\
 \varphi_{(1,p)}(p^{-s})=\zeta_{\mathrm{odd}}(\mathfrak{A}_p;s)  && \varphi_{(1,p_i)}(p_i^{-s})=\zeta_{\mathrm{odd}}(\mathfrak{A}_{p_i};s) \label{eq:Euler2}\,.
\end{eqnarray}
\begin{definition}
Let $\varphi_0(s):=\sum_{n\geq 1} \frac{b_{(0,n)}}{n^s}$, resp. $\varphi_1(s):=\sum_{n\geq 1} \frac{b_{(1,n)}}{n^s}$, be the multiplicative Dirichlet series, where $b_{(0,n)}:=a_{(0, p_{r_1}, v_{r_1})}\cdots a_{(0, p_{r_n}, v_{r_n})}$, resp. $b_{(1,n)}:=a_{(1, p_{r_1}, v_{r_1})}\cdots a_{(1, p_{r_n}, v_{r_n})}$, is the product associated to the (unique) prime decomposition $p_{r_1}^{v_{r_1}}\cdots p_{r_n}^{v_{r_n}}$ of the integer $n\geq 1$.
\end{definition}
Note that we have the following (formal) equalities
\begin{eqnarray*}
\varphi_0(s) & \stackrel{\mathrm{(a)}}{=} & \prod_{p\neq p_1, \ldots, p_m}\sum_{n\geq0} \frac{b_{(0,p^n)}}{p^{ns}}\cdot \prod_{1\leq i\leq m}(\sum_{n\geq0} \frac{b_{(0,p_i^n)}}{p_i^{ns}}) \label{eq:equality-star} \\
& = & \prod_{p\neq p_1, \ldots, p_m} \varphi_{(0,p)}(p^{-s}) \cdot \prod_{1\leq i\leq m} \varphi_{(0,p_i)}(p_i^{-s}) \\
&= &L_{\mathrm{even}}(\cA;s) \nonumber
\end{eqnarray*}
and
\begin{eqnarray*}
\varphi_1(s) & \stackrel{\mathrm{(b)}}{=} & \prod_{p\neq p_1, \ldots, p_m}\sum_{n\geq0} \frac{b_{(1,p^n)}}{p^{ns}}\cdot \prod_{1\leq i\leq m}(\sum_{n\geq0} \frac{b_{(1,p_i^n)}}{p_i^{ns}}) \label{eq:equality-star1}\\
& = & \prod_{p\neq p_1, \ldots, p_m} \varphi_{(1,p)}(p^{-s}) \cdot \prod_{1\leq i\leq m} \varphi_{(1,p_i)}(p_i^{-s}) \\
& = & L_{\mathrm{odd}}(\cA;s)\,, \nonumber
\end{eqnarray*}
where (a), resp. (b), follows from the Euler product decomposition of the (multiplicative) Dirichlet series $\varphi_0(s)$, resp. $\varphi_1(s)$. These (formal) equalities imply that the proof of Theorem \ref{thm:main-L-functions} follows now from the next result:
\begin{proposition}\label{prop:Dirichlet}
Assume that the conjectures $\mathrm{W}_{\mathrm{nc}}(\mathfrak{A}_p)$ and $\mathrm{W}_{\mathrm{nc}}(\mathfrak{A}_{p_i})$ hold. Under these assumptions, the Dirichlet series $\varphi_0(s)$, resp. $\varphi_1(s)$, converges (absolutely) in the half-plane $\mathrm{Re}(s)>1$, resp. $\mathrm{Re}(s)> \frac{3}{2}$. Moreover, $\varphi_0(s)$ and $\varphi_1(s)$ are non-zero in these half-plane regions.
\end{proposition}
\begin{proof}
Let $z> 1$, resp. $z> \frac{3}{2}$, be a real number. Thanks to the classical properties of multiplicative Dirichlet series (see \cite[Chap.~VI \S2]{Serre-book}), it suffices to show that the infinite sum $\varphi_0(z)$, resp. $\varphi_1(z)$, converges (absolutely). Note that the above (formal) equality $\varphi_0(s)=L_{\mathrm{even}}(\cA;s)$, resp. $\varphi_1(s)=L_{\mathrm{odd}}(\cA;s)$, implies that $\varphi_0(z)$, resp. $\varphi_1(z)$, converges (absolutely) if and only if the following infinite product $\prod_{p\neq p_1, \ldots, p_m} \varphi_{(0,p)}(p^{-z}) \cdot \prod^m_{i=1} \varphi_{(0,p_i)}(p_i^{-z})$, resp. the following infinite product $\prod_{p\neq p_1, \ldots, p_m} \varphi_{(1,p)}(p^{-z}) \cdot \prod^m_{i=1} \varphi_{(1,p_i)}(p_i^{-z})$, converges (absolutely). Using exponentiation, it is then enough to show that the sums $\sum_{p\neq p_1, \ldots, p_m}\phi_{(0,p)}(p^{-z})$ and $\sum_{i=1}^m \phi_{(0,p_i)}(p_i^{-z})$, resp. $\sum_{p\neq p_1, \ldots, p_m}\phi_{(1,p)}(p^{-z})$ and $\sum_{i=1}^m \phi_{(1,p_i)}(p_i^{-z})$, converge (absolutely). In what concerns the first sum, we have the following (in)equalities
\begin{eqnarray}
\sum_{p \neq p_1, \ldots, p_m} |\phi_{(0,p)}(p^{-z})|& = & \sum_{p \neq p_1, \ldots, p_m} \sum_{n \geq 1} \frac{|\#_{(0,p,n)}|}{n p^{nz}} \nonumber\\
& \leq & \sum_{p \neq p_1, \ldots, p_m} \sum_{n \geq 1} \frac{C_0}{n p^{nz}}  \label{eq:star-22}\\
& \leq & C_0 \sum_{p \neq p_1, \ldots, p_m} \sum_{n \geq 1} \frac{1}{p^{nz}} \\
& \leq & C_0 \sum_{p \neq p_1, \ldots, p_m} \frac{1}{p^{z}-1}\nonumber\,,
\end{eqnarray}
where \eqref{eq:star-22} follows from Proposition \ref{prop:bound2}. Since the series $\sum_{p \neq p_1, \ldots, p_m} \frac{1}{p^{z}-1}$ is convergent (with $z>1$), we hence conclude that $\sum_{p \neq p_1, \ldots, p_m} \phi_{(0,p)}(p^{-z})$ converges (absolutely). Similarly, we have the following (in)equalities
\begin{eqnarray}
\sum_{p \neq p_1, \ldots, p_m} |\phi_{(1,p)}(p^{-z})|& = & \sum_{p \neq p_1, \ldots, p_m} \sum_{n \geq 1} \frac{|\#_{(1,p,n)}|}{n p^{nz}} \nonumber\\
& \leq & \sum_{p \neq p_1, \ldots, p_m} \sum_{n \geq 1} \frac{C_1p^{\frac{1}{2}n}}{n p^{nz}} \label{eq:last}\\
& \leq & C_1 \sum_{p \neq p_1, \ldots, p_m} \sum_{n \geq 1} \frac{1}{p^{n(z-\frac{1}{2})}} \\
& \leq & \sum_{p \neq p_1, \ldots, p_m} \frac{1}{p^{(z-\frac{1}{2})}-1} \nonumber\,,
\end{eqnarray}
where \eqref{eq:last} follows from Proposition \eqref{prop:bound2}. Since the series $\sum_{p \neq p_1, \ldots, p_m} \frac{1}{p^{(z-\frac{1}{2})}-1}$ is convergent (with $z>\frac{3}{2}$), we hence conclude that $\sum_{p \neq p_1, \ldots, p_m} \phi_{(1,p)}(p^{-z})$ converges (absolutely). In what concerns the second sum, we have the (in)equalities
\begin{eqnarray}
\sum_{1\le i\leq m} |\phi_{(0,p_i)}(p_i^{-z})|& = & \sum_{1\leq i\leq m} \sum_{n \geq 1} \frac{|\#_{(0,p_i,n)}|}{n p_i^{nz}} \nonumber\\
& \leq & \sum_{1\leq i\leq m} \sum_{n \geq 1} \frac{C_0}{n p_i^{nz}}  \label{eq:star-222}\\
& \leq & C_0 \sum_{1\leq i\leq m} \sum_{n \geq 1} \frac{1}{p_i^{nz}} \\
& \leq & C_0 \sum_{1\leq i\leq} \frac{1}{p_i^{z}-1} \,,\nonumber
\end{eqnarray}
where \eqref{eq:star-222} follows from Proposition \ref{prop:bound2}. This implies that $\sum_{i=1}^m \phi_{(0,p_i)}(p_i^{-z})$ converges (absolutely). Similarly, we have the following (in)equalities
\begin{eqnarray}
\sum_{1\leq i\leq m} |\phi_{(1,p_i)}(p_i^{-z})|& = & \sum_{1\leq i\leq m} \sum_{n \geq 1} \frac{|\#_{(1,p_i,n)}|}{n p_i^{nz}} \nonumber\\
& \leq & \sum_{1\leq i\leq m} \sum_{n \geq 1} \frac{C_1p_i^{\frac{1}{2}n}}{n p_i^{nz}} \label{eq:last1}\\
& \leq & C_1 \sum_{1\leq i\leq m} \sum_{n \geq 1} \frac{1}{p_i^{n(z-\frac{1}{2})}} \\
& \leq & C_1 \sum_{1\leq i\leq m} \frac{1}{p_i^{(z-\frac{1}{2})}-1} \nonumber\,,
\end{eqnarray}
where \eqref{eq:last1} follows from Proposition \ref{prop:bound2}. This implies that $\sum_{i=1}^m \phi_{(1,p_i)}(p_i^{-z})$ converges (absolutely). Finally, note that the Dirichlet series $\varphi_0(s)$, resp. $\varphi_1(s)$, is non-zero in the half-plane $\mathrm{Re}(s) > 1$, resp. $\mathrm{Re}(s) > \frac{3}{2}$ because each one of its Euler factors \eqref{eq:Euler1}, resp. \eqref{eq:Euler2}, is non-zero in this half-plane region.
\end{proof}
\section{Alternative proof of Serre's convergence result}\label{sec:alternative-Serre}
Let $X$ be a smooth proper $\bbQ$-scheme of dimension $d$. Given an integer $0 \leq w \leq 2d$, consider the following $L$-function defined in \S\ref{sec:L-functions}:
\begin{equation}\label{eq:Serre-last}
L_w(X;s):=\prod_{p\neq p_1, \ldots p_m} \zeta_w(\mathfrak{X}_p;s) \cdot \prod_{1\leq i \leq m} \zeta_w(\mathfrak{X}_{p_i};s)\,.
\end{equation}
In the sixties, Serre \cite{Serre2,Serre1} proved the following result:
\begin{theorem}[Serre]\label{thm:Serre-1}
The infinite product \eqref{eq:Serre-last} converges (absolutely) in the half-plane $\mathrm{Re}(s)> \frac{w}{2}+1$. Moreover, the $L$-function $L_w(X;s)$ is non-zero in this half-plane region.
\end{theorem}
In this section, we present an alternative noncommutative proof of Theorem \ref{thm:Serre-1}. Note that by combining Theorem \ref{thm:Serre-1} with the weight decomposition \eqref{eq:L-factorization}, we conclude that the $L$-function $L(X;s)$ of $X$ converges (absolutely) in the half-plane $\mathrm{Re}(s)>d+1$. Moreover, $L(X;s)$ is non-zero in this half-plane region. 

\smallbreak

Given a prime $p\neq p_1, \ldots p_m$, let us write $\beta_{(w,p)}:=\mathrm{dim}_KH^w_{\mathrm{crys}}(\mathfrak{X}_p)$. In the same vein, given $i=1, \ldots, m$, let us write $\beta_{(w,p_i)}:=\mathrm{dim}_K H^w_{\mathrm{cyrs}}(\mathfrak{X}_{p_i})$. The next result provides a uniform upper bound for these dimensions.
\begin{proposition}\label{prop:bound11}
There exists an integer $C\gg 0$ such that $\beta_{(w,p)}\leq C$ and $\beta_{(w,p_i)}\leq C$ for every $p\neq p_1, \ldots, p_m$ and $i=1, \ldots, m$.
\end{proposition}
\begin{proof}
Consider the smooth proper $\bbQ$-linear dg category $\cA:=\perf_\dg(X)$. In this case, similarly to Example \ref{ex:sp}, we can choose for $\mathfrak{A}$ the dg category $\perf_\dg(\mathfrak{X})$ and for $\mathfrak{A}_i$ the dg category $\perf_\dg(\mathfrak{X}_i)$. Consequently, since the dg categories $\mathfrak{A}_p$ and $\mathfrak{A}_{p_i}$ are Morita equivalent to $\perf_\dg(\mathfrak{X}_p)$ and $\perf_\dg(\mathfrak{X}_{p_i})$, respectively, we conclude from Proposition \ref{prop:bound-11} that there exist integers $C_0 \gg 0$ and $C_1 \gg 0$ such that
\begin{eqnarray*}
\mathrm{dim}_K TP_0(\perf_\dg(\mathfrak{X}_p))\leq C_0 && \mathrm{dim}_K TP_1(\perf_\dg(\mathfrak{X}_p))\leq C_1\,; 
\end{eqnarray*} 
similarly with $p$ replaced by $p_i$. Thanks to \eqref{eq:Scholze1}-\eqref{eq:Scholze2}, this implies that
$$ \beta_{(w,p)}\leq \mathrm{dim}_K TP_0(\perf_\dg(\mathfrak{X}_p)) + \mathrm{dim}_K TP_1(\perf_\dg(\mathfrak{X}_p)) \leq C_0 + C_1\,;$$
similarly with $p$ replaced by $p_i$. By setting $C:=C_0+C_1$, we hence conclude that $\beta_{(w,p)}\leq C$ and $\beta_{(w,p_i)}\leq C$ for every $p\neq p_1, \ldots, p_m$ and $i=1, \ldots, m$.
\end{proof}
Given a prime $p\neq p_1, \ldots, p_m$ and an integer $n\geq 1$, consider the following integer $\#_{(w,p,n)}:=\mathrm{trace}((\mathrm{Fr}^w)^{\circ n}|H^w_{\mathrm{crys}}(\mathfrak{X}_p))$. In the same vein, given $i=1, \ldots, m$, let us consider the integer $\#_{(w,p_i,n)}:=\mathrm{trace}((\mathrm{Fr}^w)^{\circ n}|H^w_{\mathrm{crys}}(\mathfrak{X}_{p_i}))$.
\begin{proposition}\label{prop:bound-223}
There exists an integer $C\gg 0$ such that $|\#_{(w,p,n)}|\leq Cp^{\frac{w}{2}n}$, resp. $|\#_{(w,p_i,n)}|\leq Cp_i^{\frac{w}{2}n}$, for every $p\neq p_1, \ldots p_m$ and $n\geq 1$, resp. for every $i=1, \ldots, m$ and $n\geq 1$.
\end{proposition}
\begin{proof}
Given a prime $p\neq p_1, \ldots, p_m$, let us write $\{\lambda_{(w,p,1)}, \ldots, \lambda_{(w,p,\beta_{(w,p)})}\}$ for the eigenvalues (with multiplicity) of the automorphism $\mathrm{Fr}^w$ of $H^w_{\mathrm{crys}}(\mathfrak{X}_p)$. In the same vein, given $i=1, \ldots m$, let us write $\{\lambda_{(w,p_i,1)}, \ldots, \lambda_{(w,p_i,\beta_{(w,p_i)})}\}$ for the eigenvalues (with multiplicity) of the automorphism $\mathrm{Fr}^w$ of $H^w_{\mathrm{crys}}(\mathfrak{X}_{p_i})$. Under these notations, we have the following (in)equalities
\begin{eqnarray}
|\#_{(w,p,n)}| & = & |\mathrm{trace}(\mathrm{Fr}^w)| \nonumber\\
& = & |\lambda^n_{(w,p,1)} + \cdots + \lambda^n_{(w,p,\beta_{(w,p)})}| \nonumber \\
& \leq & |\lambda_{(w,p,1)}|^n + \cdots + |\lambda_{(w,p,\beta_{(w,p)})}|^n \nonumber \\ 
& \leq & \beta_{(w,p)}p^{\frac{w}{2}n} \label{eq:star-0} \\
& \leq & Cp^{\frac{w}{2}n} \label{eq:star-00}\,,
\end{eqnarray}
where \eqref{eq:star-0} follows from conjecture $\mathrm{W}(\mathfrak{X}_p)$ (proved in \cite{Deligne}) and \eqref{eq:star-00} from Proposition \ref{prop:bound11}; similarly with $p$ replaced by $p_i$.
\end{proof}
Given a prime $p\neq p_1, \ldots, p_m$, consider the following (formal) formal power series $\phi_{(w,p)}(t):=\sum_{n\geq 1} \#_{(w,p,n)}\frac{t^n}{n}$ and its exponentiation $\varphi_{(w,p)}(t):=\sum_{n\geq 0} a_{(w,p,n)}t^n$. In the same vein, given $i=1,\ldots, m$, consider the (auxiliar) formal power series $\phi_{(w,p_i)}(t):=\sum_{n\geq 1} \#_{(w,p_i,n)}\frac{t^n}{n}$ and its exponentiation $\varphi_{(w,p_i)}(t):=\sum_{n\geq 0} a_{(w,p_i,n)}t^n$. Note that thanks to the above Lemma \ref{lem:exercise1}, we have the (formal) equalities:
\begin{eqnarray}\label{eq:Euler-Serre}
\varphi_{(w,p)}(p^{-s}) = \zeta_w(\mathfrak{X}_p;s) && \varphi_{(w,p_i)}(p_i^{-s}) = \zeta_w(\mathfrak{X}_{p_i};s)\,.
\end{eqnarray}
\begin{definition}
Let $\varphi_w(s):=\sum_{n\geq 1} \frac{b_{(w,n)}}{n^s}$ be the Dirichlet series, where $b_{(w,n)}:=a_{(w, p_{r_1}, v_{r_1})}\cdots a_{(w, p_{r_n}, v_{r_n})}$ is the product associated to the (unique) prime decomposition $p_{r_1}^{v_{r_1}}\cdots p_{r_n}^{v_{r_n}}$ of the integer $n\geq 1$.
\end{definition}
Note that we have the following (formal) equalities
\begin{eqnarray*}
\varphi_w(s) & \stackrel{\mathrm{(a)}}{=} & \prod_{p\neq p_1, \ldots, p_m}(\sum_{n\geq0} \frac{b_{(w,p^n)}}{p^{ns}})\cdot \prod_{1\leq i\leq m}(\sum_{n\geq0} \frac{b_{(w,p_i^n)}}{p_i^{ns}}) \label{eq:last-111} \\
& = & \prod_{p\neq p_1, \ldots, p_m} \varphi_{(w,p)}(p^{-s}) \cdot \prod_{1\leq i\leq m} \varphi_{(w,p_i)}(p_i^{-s}) \\
& = & L_w(X;s) \nonumber\,,
\end{eqnarray*}
where (a) follows from the Euler product decomposition of the (multiplicative) Dirichlet series $\varphi_w(s)$. These (formal) equalities imply that the proof of Theorem \ref{thm:Serre-1} follows now from the next result:
\begin{proposition}
The Dirichlet series $\varphi_w(s)$ converges (absolutely) in the half-plane $\mathrm{Re}(s)> \frac{w}{2} +1$. Moreover, $\varphi_w(s)$ is non-zero in this half-plane region. 
\end{proposition}
\begin{proof}
Let $z> \frac{w}{2}+1$ be a real number. Similarly to the proof of Proposition \ref{prop:Dirichlet}, it is enough to show that the sums $\sum_{p\neq p_1, \ldots, p_m}\phi_{(w,p)}(p^{-z})$ and $\sum_{i=1}^m \phi_{(w,p_i)}(p_i^{-z})$ converge (absolutely). Note that we have the following (in)equalities
\begin{eqnarray}
\sum_{p \neq p_1, \ldots, p_m} |\phi_{(w,p)}(p^{-z})|& = & \sum_{p \neq p_1, \ldots, p_m} \sum_{n \geq 1} \frac{|\#_{(w,p,n)}|}{n p^{nz}} \nonumber\\
& \leq & \sum_{p \neq p_1, \ldots, p_m} \sum_{n \geq 1} \frac{Cp^{\frac{w}{w}n}}{n p^{nz}}  \label{eq:star-222-3}\\
& \leq & C \sum_{p \neq p_1, \ldots, p_m} \sum_{n \geq 1} \frac{1}{p^{n(z-\frac{w}{2})}} \\
& \leq & C \sum_{p \neq p_1, \ldots, p_m} \frac{1}{p^{(z-\frac{w}{2})}-1} \nonumber \,,
\end{eqnarray}
where \eqref{eq:star-222-3} follows from Proposition \ref{prop:bound-223}. Since the series $\sum_{p \neq p_1, \ldots, p_m} \frac{1}{p^{(z-\frac{w}{2})}-1}$ is convergent (with $z>\frac{w}{2}+1$), we hence conclude that $\sum_{p \neq p_1, \ldots, p_m} \phi_{(w,p)}(p^{-z})$ converges (absolutely). Similarly, we have the following (in)equalities
\begin{eqnarray}
\sum_{1\leq i\leq m} |\phi_{(w,p_i)}(p_i^{-z})|& = & \sum_{1\leq i\leq m} \sum_{n \geq 1} \frac{|\#_{(w,p_i,n)}|}{n p_i^{nz}} \nonumber\\
& \leq & \sum_{1\leq i\leq m} \sum_{n \geq 1} \frac{Cp_i^{\frac{w}{2}n}}{n p_i^{nz}}  \label{eq:star-2222}\\
& \leq & C\sum_{1\leq i\leq m} \sum_{n \geq 1} \frac{1}{p_i^{n(z-\frac{w}{2})}} \\
& \leq & C \sum_{1\leq i\leq m} \frac{1}{p_i^{(z-\frac{w}{2})}-1} \nonumber \,,
\end{eqnarray}
where \eqref{eq:star-2222} follows from Proposition \ref{prop:bound-223}. This implies that $\sum_{i=1}^m \phi_{(w,p_i)}(p_i^{-z})$ converges (absolutely). Finally, note that $\varphi_w(s)$ is non-zero in the half-plane $\mathrm{Re}(s) > \frac{w}{2}+1$ because its Euler factors \eqref{eq:Euler-Serre} are non-zero in this region.
\end{proof}

\section{Proof of Theorem \ref{thm:implications}}
Recall first that, thanks to the Hochschild-Kostant-Rosenberg theorem (consult \cite{FT}), we have the following natural isomorphisms of finite-dimensional $\bbQ$-vector spaces $HP_0(\perf_\dg(X))\simeq \bigoplus_{w\,\mathrm{even}}H_{dR}^w(X)$ and $HP_1(\perf_\dg(X))\simeq  \bigoplus_{w\,\mathrm{odd}}H_{dR}^w(X)$, where $H_{dR}^\ast(X)$ stands for the de Rham cohomology of $X$. Recall also that we have the following classical isomorphism 
\begin{eqnarray*}
K_0(\perf_\dg(X))_\bbQ \stackrel{\simeq}{\too} \bigoplus_{0\leq i\leq d} \cZ^i(X)_\bbQ/_{\!\sim\mathrm{rat}} && [\cF] \mapsto \mathrm{chern}(\cF) \cdot \sqrt{\mathrm{td}}_X\,,
\end{eqnarray*}
where $d:=\mathrm{dim}(X)$, $\mathrm{chern}(\cF)$ stands for the Chern character of $\cF$, and $\sqrt{\mathrm{td}}_X$ stands for the square root of the Todd class of $X$; see \cite[\S18.3]{Fulton}. Under the above isomorphisms, the $\bbQ$-linear homomorphism $\mathrm{ch}\colon K_0(\perf_\dg(X))_\bbQ \to HP_0(\perf_\dg(X))$ defined in \S\ref{sub:homological} corresponds to the classical cycle class map (with values in de Rham cohomology). Consequently, we obtain an induced isomorphism
\begin{equation}\label{eq:homo1}
K_0(\perf_\dg(X))_\bbQ/_{\!\sim\mathrm{hom}} \simeq \bigoplus_{0\leq i\leq d} \cZ^i(X)_\bbQ/_{\!\sim \mathrm{hom}} \end{equation}
as well as an induced isomorphism:
\begin{equation}\label{eq:homo2}
K_0(\perf_\dg(X))^0_\bbQ \simeq \bigoplus_{0\leq i\leq d} \cZ^i(X)^0_\bbQ/_{\!\sim\mathrm{rat}} = \bigoplus_{1\leq i\leq d} \cZ^i(X)^0_\bbQ/_{\!\sim\mathrm{rat}}\,.
\end{equation}

The first implication of Theorem \ref{thm:implications} is a consequence of the equalities
\begin{eqnarray}
\mathrm{ord}_{s=1}L_{\mathrm{even}}(\perf_\dg(X);s) & = & \sum_{w\,\mathrm{even}} \mathrm{ord}_{s=1}L_w(X;s+ \frac{w}{2})\label{eq:star1} \\
& = & \sum_{w\,\mathrm{even}}\mathrm{ord}_{s=\frac{w}{2}+1} L_w(X;s) \nonumber \\
& = & - \sum_{w\, \mathrm{even}} \mathrm{dim}_\bbQ \cZ^{\frac{w}{2}}(X)_\bbQ/_{\!\sim\mathrm{hom}} \label{eq:star2} \\
& = & - \mathrm{dim}_\bbQ K_0(\perf_\dg(X))_\bbQ/_{\!\sim \mathrm{hom}}\,, \label{eq:star3}
\end{eqnarray}
where \eqref{eq:star1} follows from factorization \eqref{eq:equality-L1} (and from Remark \ref{rk:key-1}), \eqref{eq:star2} from Beilinson's conjecture $\mathrm{B}_w^{\frac{w}{2}+1}(X)$, and \eqref{eq:star3} from isomorphism \eqref{eq:homo1}. Similarly, the second implication of Theorem \ref{thm:implications} is a consequence of the equalities  
\begin{eqnarray}
\mathrm{ord}_{s=1}L_{\mathrm{odd}}(\perf_\dg(X);s) & = & \sum_{w\,\mathrm{odd}} \mathrm{ord}_{s=1}L_w(X;s+ \frac{w-1}{2})\label{eq:star11} \\
& = & \sum_{w\,\mathrm{odd}}\mathrm{ord}_{s=\frac{w+1}{2}} L_w(X;s) \nonumber \\
& = & \sum_{w\, \mathrm{odd}} \mathrm{dim}_\bbQ \cZ^{\frac{w+1}{2}}(X)^0_\bbQ/_{\!\sim\mathrm{rat}} \label{eq:star22} \\
& = & \mathrm{dim}_\bbQ K_0(\perf_\dg(X))^0_\bbQ\,, \label{eq:star33}
\end{eqnarray}
where \eqref{eq:star11} follows from factorization \eqref{eq:equality-L2} (and from Remark \ref{rk:key-1}), \eqref{eq:star22} from Beilinson's conjecture $\mathrm{B}_w^{\frac{w+1}{2}}(X)$, and \eqref{eq:star33} from isomorphism \eqref{eq:homo2}. 

Assume now that the Beilinson-Soul\'e vanishing conjecture holds. Under this assumption, the third implication of Theorem \ref{thm:implications} is a consequence of the equalities
\begin{eqnarray}
\mathrm{ord}_{s=0}L_{\mathrm{even}}(\perf_\dg(X);s) & = & \sum_{w\,\mathrm{even}} \mathrm{ord}_{s=0}L_w(X;s+ \frac{w}{2})\label{eq:star111} \\
& = & \sum_{w\,\mathrm{even}}\mathrm{ord}_{s=\frac{w}{2}} L_w(X;s) \nonumber \\
& = & \sum_{w\, \mathrm{even}} \mathrm{dim}_\bbQ H^{w+1}_{\mathrm{mot}}(X;\bbQ(\frac{w}{2}+1)) \label{eq:star222} \\
& = & \mathrm{dim}_\bbQ K_1(\perf_\dg(X))_\bbQ \label{eq:star333}\,,
\end{eqnarray}
where \eqref{eq:star111} follows from factorization \eqref{eq:equality-L1} (and from Remark \ref{rk:key-1}), \eqref{eq:star222} from Beilinson's conjecture $\mathrm{B}^{\frac{w}{2}}_w(X)$, and \eqref{eq:star333} from Proposition \ref{prop:final} below. Similarly, the fourth implication of Theorem \ref{thm:implications} is a consequence of the equalities
\begin{eqnarray}
\mathrm{ord}_{s=-1}L_{\mathrm{even}}(\perf_\dg(X);s) & = & \sum_{w\,\mathrm{even}} \mathrm{ord}_{s=-1}L_w(X;s+ \frac{w}{2})\label{eq:star1111} \\
& = & \sum_{w\,\mathrm{even}}\mathrm{ord}_{s=\frac{w}{2}-1} L_w(X;s) \nonumber \\
& = & \sum_{w\, \mathrm{even}} \mathrm{dim}_\bbQ H^{w+1}_{\mathrm{mot}}(X;\bbQ(\frac{w}{2}+2)) \label{eq:star2222} \\
& = & \mathrm{dim}_\bbQ K_3(\perf_\dg(X))_\bbQ \label{eq:star3333}\,,
\end{eqnarray}
where \eqref{eq:star1111} follows from factorization \eqref{eq:equality-L1} (and from Remark \ref{rk:key-1}), \eqref{eq:star2222} from Beilinson's conjecture $\mathrm{B}_w^{\frac{w}{2}-1}(X)$, and \eqref{eq:star3333} from Proposition \ref{prop:final}. 

Finally, the last implication of Theorem \ref{thm:implications} is a consequence of the equalities
\begin{eqnarray}
\mathrm{ord}_{s=0}L_{\mathrm{odd}}(\perf_\dg(X);s) & = & \sum_{w\,\mathrm{odd}} \mathrm{ord}_{s=0}L_w(X;s+ \frac{w-1}{2})\label{eq:star11111} \\
& = & \sum_{w\,\mathrm{odd}}\mathrm{ord}_{s=\frac{w-1}{2}} L_w(X;s) \nonumber \\
& = & \sum_{w\, \mathrm{odd}} \mathrm{dim}_\bbQ H^{w+1}_{\mathrm{mot}}(X;\bbQ(\frac{w+3}{2})) \label{eq:star22222} \\
& = & \mathrm{dim}_\bbQ K_2(\perf_\dg(X))_\bbQ \label{eq:star33333}\,,
\end{eqnarray}
where \eqref{eq:star11111} follows from factorization \eqref{eq:equality-L2} (and from Remark \ref{rk:key-1}), \eqref{eq:star22222} from Beilinson's conjecture $\mathrm{B}_w^{\frac{w-1}{2}}(X)$, and \eqref{eq:star33333} from Proposition \ref{prop:final}.
\begin{proposition}\label{prop:final}
Let $X$ be a smooth proper $\bbQ$-scheme of dimension $d$ and $n>0$ an integer. Assuming the Beilinson-Soul\'e vanishing conjecture, we have the following direct sum decomposition(s) of rational algebraic $K$-theory:
$$ K_n(\perf_\dg(X))_\bbQ \simeq \bigoplus_{\frac{n}{2}<r \leq d+n} H_{\mathrm{mot}}^{2r-n}(X;\bbQ(r)) = \bigoplus_{\frac{n}{2}<r \leq d+n-1} H_{\mathrm{mot}}^{2r-n}(X;\bbQ(r))\,.$$
\end{proposition}
\begin{proof}
Recall first that the algebraic $K$-theory groups $K_n(\perf_\dg(X))_\bbQ$ of the dg category $\perf_\dg(X)$ are naturally isomorphic to the classical algebraic $K$-theory group $K_n(X)_\bbQ$ of $X$; consult \cite[\S2.2.1]{book}. As proved in \cite[\S2.8]{Soule}, we have the direct sum decomposition $K_n(X)_\bbQ \simeq \bigoplus_{r=0}^{d+n} K_n(X)_\bbQ^{(r)}$, where $K_n(X)_\bbQ^{(r)}$ stands for the $r^{\mathrm{th}}$ eigenspace with respect to the Adams operations on $K_n(X)_\bbQ$. Making use of the classical identification $H^i_{\mathrm{mot}}(X;\bbQ(j))\simeq K_{2j-i}(X)_\bbQ^{(j)}$, we hence conclude that
\begin{equation}\label{eq:iso}
K_n(X)_\bbQ \simeq \bigoplus_{0\leq r\leq d+n} H^{2r-n}_{\mathrm{mot}}(X;\bbQ(r))\,.
\end{equation}
It is well-known that $H_{\mathrm{mot}}^{2d+n}(X;\bbQ(d+n))$ is isomorphic to the cokernel of the residue map $\partial\colon \bigoplus_{x \in X^{(d-1)}}K^M_{n+1}(\kappa(x))_\bbQ \to \bigoplus_{x \in X^{(d)}}K^M_n(\kappa(x))_\bbQ$, where $X^{(d)}$ stands for ~the set of points of codimension $d$ on $X$ and $K^M_n(\kappa(x))$ for the $n^{\mathrm{th}}$ Milnor $K$-theory group of the residue field $\kappa(x)$. Since $\kappa(x)$ is a number field, the Milnor $K$-theory groups $K^M_n(\kappa(x))$ are torsion. This implies that $H^{2d+n}_{\mathrm{mot}}(X;\bbQ(d+n))$ is zero.

Assuming the Beilinson-Soul\'e vanishing conjecture, all the motivic cohomology groups $H^{2r-n}_{\mathrm{mot}}(X;\bbQ(r))$, with $0 \leq r \leq \lfloor\frac{n}{2}\rfloor$, are zero. Consequently, since $H^{2d+n}_{\mathrm{mot}}(X;\bbQ(d+n))$ is also zero, we obtain the sought decomposition(s).
\end{proof}
\begin{remark}[Potential generalization]\label{rk:generalization2}
Let $X$ be a smooth proper $\bbQ$-scheme of dimension $d$. Assuming the Beilinson conjecture and the Beilinson-Soul\'e vanishing conjecture, note that the above proof of Theorem \ref{thm:implications} (and of Proposition \ref{prop:final}) shows that the following equalities hold
\begin{eqnarray*}
\mathrm{ord}_{s=j}L_{\mathrm{even}}(\perf_\dg(X);s)= \mathrm{dim}_\bbQ K_{1-2j}(\perf_\dg(X))_\bbQ && j \leq -2 \\
\mathrm{ord}_{s=j}L_{\mathrm{odd}}(\perf_\dg(X);s)= \mathrm{dim}_\bbQ K_{2-2j}(\perf_\dg(X))_\bbQ && j \leq -1
\end{eqnarray*}
if and only if the groups $\{H_{\mathrm{mot}}^{2r+2j-1}(X;\bbQ(r)\,|\,d-j+1< r \leq d-2j\}$, resp. $\{H_{\mathrm{mot}}^{2r+2j-2}(X;\bbQ(r))\,|\,d-j+1<r \leq d-2j+1\}$, are zero. Unfortunately, to the best of the author's knowledge, nothing is known about these groups.
\end{remark}

\medbreak\noindent\textbf{Acknowledgments:} I am very grateful to Maxim Kontsevich for an enlightening discussion about the noncommutative Weil conjecture during the {\em Mathematische Arbeitstagung} 2017 (in honor of Yuri I. Manin), and for numerous comments, suggestions and corrections on the previous version of this article. I am also thankful to Niranjan Ramachandran for discussions concerning zeta and $L$-functions, to Dmitry Kaledin for discussions concerning functional equations, and to Pieter Belmans for comments on the original previous version of this article.

\end{document}

\end{proof}